\documentclass[fleqno,12pt]{amsart}
\usepackage{xcolor}
\usepackage[draft]{todonotes}
\usepackage[english]{babel}
\usepackage{geometry}
\usepackage[T1]{fontenc}
\usepackage[latin1]{inputenc}
\usepackage{amsmath,amssymb,mathrsfs,amsthm, tikz-cd,mathrsfs}
\usepackage{soul}
\usepackage{epsfig}
\usepackage{hyperref}
\usepackage{float}
\usepackage{graphicx}
\usepackage{enumitem}
\usepackage{datetime}
\usepackage{xcolor}
\usepackage{fancyhdr}

\usepackage{comments}
\newComments\AL{AL}{red}
\newComments\SB{SB}{red}

\newtheorem{Theorem}{Theorem}[section]
\newtheorem{Proposition}[Theorem]{Proposition}

\newtheorem{Lemma}[Theorem]{Lemma}

\newtheorem{Corollary}[Theorem]{Corollary}
\newtheorem{Definition}[Theorem]{Definition}
\newtheorem{Example}[Theorem]{Example}
\newtheorem{Remark}[Theorem]{Remark}

\newcommand{\fg}{\mathfrak{g}}
\newcommand{\fh}{\mathfrak{h}}

\newcommand {\fosp}{\mathfrak{osp}}

\newcommand {\fder}{\mathfrak{der}}

\newcommand {\Span}{\mathrm{Span}}

\newcommand {\ad}{\mathrm{ad}}

\newcommand {\fgl}{\mathfrak{gl}}

\newcommand {\fq}{\mathfrak{q}}

\newcommand {\K}{\mathbb{K}}

\newcommand {\ev} {{\bar0}}
\newcommand {\od} {{\bar1}}

\begin{document}

\title{Hom-Lie superalgebras  in characteristic 2}


\author{Sofiane Bouarroudj}

 \address{Division of Science and Mathematics, New York University Abu Dhabi, Po Box 129188, Abu Dhabi, United Arab Emirates.}

\email{sofiane.bouarroudj@nyu.edu}

\author{Abdenacer Makhlouf}
\address{IRIMAS-Mathematics Department, University of Haute Alsace, Mulhouse, France.}

\email{abdenacer.makhlouf@uha.fr }


\thanks{SB was supported by the grant NYUAD-065.}

            \date{\today, \currenttime}

\begin{abstract}
The main goal of this paper is to develop the structure theory of Hom-Lie superalgebras in characteristic 2. We discuss their representations, semidirect product, $\alpha^k$-derivations and provide a classification in low dimension.
We introduce another notion of  restrictedness on Hom-Lie algebras in characteristic~2, different from one given by Guan and Chen. This  definition is inspired by the process of queerification of restricted Lie algebras in characteristic 2. We also show that any restricted Hom-Lie algebra in characteristic 2 can be queerified to  give rise to a Hom-Lie superalgebra. 
Moreover, we develop a cohomology theory of Hom-Lie superalgebras in characteristic 2, which provides a cohomology of ordinary Lie superalgebras. Furthermore, we  establish a  deformation theory of Hom-Lie superalgebras in characteristic 2 based on this cohomology.  
\smallskip

\end{abstract}

\maketitle

\tableofcontents

\section{Introduction} \label{SecDef}
Throughout the text, $\mathbb{K}$ stands for an arbitrary field of characteristic 2. In almost all our constructions, $\mathbb{K}$ is arbitrary. There are a few instances where $\mathbb{K}$ is  required to be infinite. We will point out at these instances. 
\subsection{Lie superalgebras in characteristic 2} Roughly speaking, a Lie superalgebra in characteristic 2 is a $\mathbb Z/2$-graded vector space that has a Lie algebra structure on 
the even part, and endowed with a squaring on the odd part that satisfies a modified Jacobi identity, see \S \ref{defp=2} for a precise definition. Because we are in characteristic 2, those Lie superalgebras are sometimes being confused with ${\mathbb Z}/2$-graded Lie algebras, though they are totally different algebras due to the presence of the squaring. They can, however, be considered as $\mathbb Z/2$-graded Lie algebra by forgetting the super structure. The other way round is not always true in general.

The classification of simple Lie superalgebras in characteristic 2 is still an open and wide problem. Nevertheless, Lie superalgebras in characteristic 2 admitting a Cartan matrix have been classified in \cite{BGL}, with the following assumption: each Lie superalgebra posses a Dynkin diagram with only one odd node. The list of non-equivalent Cartan matrices for each Lie superalgebra is also listed in \cite{BGL}. Moreover, it has recently been showed in \cite{BLLSq} that each finite-dimensional simple Lie superalgebra in characteristic 2 can be obtained from a simple finite-dimensional Lie algebra in characteristic 2 by one of two methods, hence reducing the classification to the classification of simple Lie algebras which in its own is a very tough problem. As a matter of fact, there are plenty of (vectorial and non-vectorial) Lie superalgebras in characteristic 2 that have no analogue in other characteristics, see \cite{BeBou, BGLLS, BLLSq, BLS, LeD}. 

It is worth mentioning that the characteristic 2 case is a very tricky case, due to the presence of the squaring. It does require new ideas and techniques. 
\subsection{Hom-Lie superalgebras in characteristic 2}
The first instances of Hom-type algebras appeared in physics literature, see for example \cite{CLK}, where $q$-deformations of some Lie algebras of vector fields led to a structure in which Jacobi identity is no longer satisfied. This class of algebras were formalized and studied in \cite{HLS,LS}, where  they were called {\it Hom-Lie algebras} since the Jacobi identity is  twisted by a homomorphism. The super case were considered in \cite{AM}, where Hom-Lie superalgebras were introduced  as a $\mathbb{Z}/2$-graded generalization of the Hom-Lie algebras. The authors of \cite{AM}  characterized  Hom-Lie admissible superalgebras and proved a $\mathbb{Z}/2$-graded version of a Hartwig-Larsson-Silvestrov Theorem which led to a  construction of  a $q$-deformed Witt superalgebra using $\sigma$-derivations. Moreover, they derived a one parameter family of Hom-Lie  superalgebras deforming the orthosymplectic Lie superalgebra $\fosp(1|2)$. The cohomology of Hom-Lie superalgebras was defined in \cite{AMS}. Notice that all these studies and results were performed over a field of characteristic 0. 
\subsection{The main results}
The main purpose of this this paper is to tackle the positive characteristic and provide a study of Hom-Lie superalgebras in characteristic 2. We introduce the main definitions and some key constructions as well as a cohomology theory fitting with a deformation theory. In Section 2, we recall some basic definitions,  and introduce Hom-Lie algebras and Hom-Lie superalgebras over fields of  characteristic 2 and some related structures. We show that a Lie superalgebra in characteristic 2 and an even Lie algebra morphism give rise to a Hom-Lie superalgebra in characteristic 2. Moreover, we provide a classification of Hom-Lie superalgebras in characteristic 2 in low dimensions. In Section 3, we consider representations and semidirect product of Hom-Lie superalgebras in characteristic 2. The structure map defining a Hom-Lie superalgebra in characteristic 2 allows  a new type of derivations called $\alpha^k$-derivations discussed in Section 4. In Section 5, we introduce the notion of $p$-structure and discuss  queerification of restricted Hom-Lie algebras in characteristic 2. Section 6 is dedicated to cohomology theory. We construct a cohomology complex of a Hom-Lie superalgebra  $\fg$ in characteristic 2 with values in a $\mathfrak{g}$-module. This cohomology complex has no analogue in characteristic $p\not= 2$. In the last section, we provide a deformation theory of Hom-Lie superalgebras in characteristic 2 using the cohomology we constructed.
\section{Backgrounds and main definitions} \label{MainDef}
Let $V$ and $W$ be two vector spaces over $\mathbb{K}$. A map $s:V\rightarrow W$ is called a {\it squaring} if 
\begin{equation}
    \label{squaring}
    s(\lambda x)=\lambda^2 s(x)\quad \text{ for all $\lambda\in \mathbb{K}$ and for all $x\in V$.}
\end{equation}
\subsection{Lie superalgebras in characteristic 2.}\label{defp=2}
Following \cite{LeD}, a \textit{Lie superalgebra} in characteristic 2 is a
superspace $\fg=\fg_\ev\oplus\fg_\od$ over $\mathbb{K}$ such that $\fg_\ev$ is an ordinary Lie algebra, $\fg_\od$ is a
$\fg_\ev$-module made two-sided by symmetry, and on $\fg_\od$ a squaring, denoted by $s_\fg:\fg_\od\rightarrow \fg_\ev$, is
given. The bracket on $\fg_\ev$, as well as the action of $\fg_\ev$ on $\fg_\od$, is denoted by the same symbol $[\cdot ,\cdot]$. For any $x, y\in\fg_\od$, their bracket is then defined by
\[
[x,y]:=  s(x+y)-s(x)-s(y).
\]
The bracket is extended to non-homogeneous elements by bilinearity.
The Jacobi identity involving the squaring reads as follows:
\begin{equation*}
[s(x),y]=[x,[x,y]]\;\text{ for any $x\in \fg_\od$ and $y\in\fg$}.\label{JIS}
\end{equation*}
Such a Lie superalgebra in characteristic 2 will be denoted by $(\fg, [\cdot,\cdot], s)$.

For any Lie superalgebra $\fg$ in characteristic 2, its \textit{derived algebras} are
defined to be (for $i\geq 0$)
\[
\fg^{(0)}: =\fg, \quad
\fg^{(i+1)}=[\fg^{(i)},\fg^{(i)}]+\Span\{s(x)\mid x\in (\fg^{(i)})_\od\}.
\]
A linear map $D:\fg\rightarrow \fg$ is called a \textit{derivation} of the Lie superalgebra $\fg$ if, in addition to
\begin{eqnarray}
\label{Der1} D([x,y])&=&[D(x),y]+[x,D(y)]\quad \text{for any $x\in \fg_\ev$ and $y\in \fg$ we have}\\[2mm]
\label{Der2} D(s(x))&=&[D(x),x]\quad \text{for any $x\in \fg_\od$}.
\end{eqnarray}
It is worth noticing that condition (\ref{Der2}) implies condition (\ref{Der1}) if $x,y\in \fg_\od$. 

We denote the space of all derivations of $\fg$ by $\fder (\fg)$. 

Let $(\fg,[\cdot ,\cdot]_\fg,s_\fg)$ and $(\fh,[\cdot ,\cdot]_\fh,s_\fh)$ be two Lie superalgebras in characteristic 2. An even linear map $\varphi:\fg\rightarrow \fh$ is called a \textit{morphism} (of Lie superalgebras) if, in addition to
\begin{eqnarray*}
\label{Hom1} \varphi([x,y]_\fg)&=&[\varphi(x), \varphi (y)]_\fh\quad \text{for any $x\in \fg_\ev$ and $y\in \fg$, we have}\\[2mm]
\label{Hom2} \varphi(s_\fg(x))&=& s_\fh(\varphi(x))\quad \text{for any $x\in \fg_\od$}.
\end{eqnarray*}

Therefore, morphisms in the category of Lie superalgebras in characteristic 2 preserve not only the bracket but the squaring as well. In particular, subalgebras and ideals have to be stable under the bracket and the squaring. 

An even linear map $\rho: \fg\rightarrow \fgl(V)$ is a \textit{representation of the Lie superalgebra} $(\fg, [\cdot,\cdot],s)$ in the superspace $V$ called \textit{$\fg$-module} if
\begin{equation}\label{repres}
\begin{array}{l}
\rho([x, y])=[\rho (x), \rho(y)]\quad \text{ for any $x, y\in
\fg$}; \text{ and }
\rho (s(x))=(\rho (x))^2\text{ for any $x\in\fg_\od$.}
\end{array}
\end{equation}
\begin{Remark}
{\rm Associative superalgebras in characteristic 2 leads to Lie superalgebras in characteristic 2. The bracket is standard and the squaring is defined by $
s(x)=x\cdot x,$ for every odd element $x$.}
\end{Remark}
\subsection{Hom-Lie algebras in characteristic 2} A \textit{Hom-Lie algebra} in characteristic 2 is a vector space $\fg$ over $\mathbb{K}$ and a map $\alpha \in \mathrm{End}(\fg)$  together with a bracket satisfying the following conditions:
\[
\begin{array}{c}
[x,x]=0, \;\;\; \alpha[x,y]=[\alpha(x),\alpha(y)] \;\;  \text{ and } \;\;  [\alpha(x),[y,z]]+\circlearrowleft (x,y,z)=0, \quad \text{for all $x,y, z\in \fg$.}
\end{array}
\]
Such a Hom-Lie algebra will be denoted by $(\fg, [\cdot,\cdot],\alpha)$.

A \textit{representation of a Hom-Lie algebra} $(\fg , [\cdot , \cdot ]_\fg,\alpha )$ is a triplet $(V, [\cdot,\cdot]_V, \beta)$ where $V$ is a vector space, $\beta \in \fgl(V)$, and  $[\cdot,\cdot]_V$ is the action of $\fg$ on $V$  such that (for all $x,y\in \fg$ and for all $v\in V$):
\begin{equation}\label{Hom-Lie-repr2}
\begin{array}{l}
[\alpha(x), \beta(v)]_V=\beta([x,v]_V),  \quad [[x,y]_\fg, \beta(v)]_V=[\alpha(x), [y,v]_V]_V+ [\alpha(y), [x,v]_V]_V.
\end{array}
\end{equation}
Writing Eq. (\ref{Hom-Lie-repr2}) using the notation of Eq. (\ref{repres}), we put $\rho_\beta:=[\cdot,\cdot]_V$ and obtain (for all $x,y \in \fg$):
\begin{equation}\label{Hom-Lie-repr2V2}
\begin{array}{l}
\rho_\beta(\alpha(x)) \circ \beta=
\beta \circ \rho (x),  \quad \rho([x,y]_\fg) \circ \beta=\rho(\alpha(x))\rho(y)+
\rho(\alpha(y))\rho(x).
\end{array}
\end{equation}
\subsection{Hom-Lie superalgebras in characteristic 2} 
Our main definition is given below. Due to the presence of the squaring, our approach to define Hom-Lie superalgebras in characteristic 2 will differ from that used in  characteristics $p\not =2$, see \cite{AM}. 
\begin{Definition}\label{MainDef} \rm{ 
A \textit{Hom-Lie superalgebra} in characteristic 2 is a quadruple $( \fg, [\cdot,\cdot],s, \alpha)$  consisting of a
$\mathbb{Z}/2$-graded superspace $\fg=\fg_\ev\oplus\fg_\od$ over $\mathbb{K}$, a symmetric bracket $ [\cdot,\cdot]$,  a squaring $s:\fg_\od \rightarrow \fg_\ev$, and an even map $\alpha \in \mathrm{End}(\fg)$  such that
\begin{enumerate} 
\item[\textup{(}i\textup{)}] $(\fg_\ev,[\cdot ,\cdot], \alpha|_{\fg_\ev})$ is an ordinary Hom-Lie algebra, 
\item[\textup{(}ii\textup{)}] $\fg_\od$ is a
$\fg_\ev$-module made two-sided by symmetry,  where the action is still denoted by the bracket $ [\cdot,\cdot]$.
\item[\textup{(}iii\textup{)}] the map 
\begin{equation} \fg_\od\times \fg_\od \rightarrow \fg_\ev \quad (x,y)\mapsto       s(x+y) + s(x) + s(y)
\end{equation}
 is  bilinear, and induces  the bracket on odd elements; namely,  for any $x, y\in\fg_\od$:
\[
[x,y]:=  s(x+y)+s(x)+s(y), 
\]
\item[\textup{(}iv\textup{)}] the following three conditions hold
\begin{eqnarray}\label{JISH} 
[s(x),\alpha(y)] & = &[\alpha(x),[x,y]]\;\text{ for any $x\in \fg_\od$ and $y\in\fg$}, \\ 
\label{multiplicativity} \alpha([x,y]) & = &[\alpha(x),\alpha(y)] \;\text{ for any } x,y\in \fg,\\
\label{multsq} \alpha(s(x)) & = & s(\alpha(x))\; \text{ for any } x \in \fg_\od.
\end{eqnarray} 
\end{enumerate}
}
\end{Definition}
\begin{Remark} {\rm 
\textup{(}i\textup{)} The Jacobi identity on triples in $\{\fg_0,\fg_1,\fg_1\}$ and  $\{\fg_1,\fg_1,\fg_1\}$ follow from condition \ref{JISH}. We, therefore, recover the usual definition of Hom-Lie superalgebras \cite{AM}.

\textup{(}ii\textup{)} Since we work over of field of characteristic 2, skewsymmetry and symmetry coincide  since $ -1 \equiv 1\, (\mod 2)$.
 
\textup{(}iii\textup{)} We may want to consider Hom-Lie superalgebras in characteristic 2 without conditions \eqref{multiplicativity} and \eqref{multsq}, which corresponds to the multiplicativity of the structure map $\alpha$.
}
\end{Remark}

Let  $( \fg, [\cdot,\cdot]_\fg,s_\fg, \alpha)$ and $( \fg', [\cdot,\cdot]_{\fg'},s_{\fg'}, \alpha')$ be two Hom-Lie superalgebras in characteristic 2.  A  map $\phi: \fg\rightarrow \fg'$ is a morphism of Hom-Lie superalgebras if the following conditions are satisfied:   
\begin{equation}
\label{homphialpha}
\phi(  [\cdot, \cdot]_\fg)=[\phi(\cdot),\phi(\cdot)]_{\fg'}, \quad \phi \circ s_\fg=s_{\fg'}\circ \phi,  \quad \phi\circ \alpha=\alpha'\circ \phi.
\end{equation}
Two Hom-Lie superalgebras $( \fg, [\cdot,\cdot]_\fg,s_\fg, \alpha)$ and $( \fg', [\cdot,\cdot]_{\fg'},s_{\fg'}, \alpha')$ are called  \emph{isomorphic} if there exists a homomorphism $\phi:\fg \rightarrow \fg'$ as in (\ref{homphialpha}) that it is bijective.\\

Let  $( \fg, [\cdot,\cdot],s, \alpha)$ be a Hom-Lie superalgebra in characteristic 2. Let $I$ be a subset  of $\fg$. The set $I$ is called an ideal of $\fg$ if and only if $I$ is closed under addition and scalar multiplication, together with
\[
[I,\fg]\subseteq I, \;  \alpha(I)\subseteq I \text{ and $s(x)\in I$ whenever $x\in I\cap \fg_\od$.}
\]
In particular, if the ideal $I$ is homogeneous; namely $I=I\cap \fg_\ev\oplus I\cap\fg_\od=I_\ev\oplus I_\od$ then the condition involving the squaring reads $s(x)\in I_\ev$ for all $x\in I_\od$. In addition, the superspace $\fg/I$ is also a Hom-Lie superalgebra in characteristic 2. The bracket and the squaring are defined as follows:
\[
\begin{array}{rcll}
[x +I, y+I]&:=&[x,y]+I & \text{for all $x,y\in \fg$,}\\[2mm]
s(x+I) & := & s(x)+I & \text{ for all $x\in \fg_\od$,}
\end{array}
\]
while the twist map $\tilde \alpha$ on $\fg/I$ is defined by
\[
\tilde \alpha(x+I)=\alpha(x)+I \quad \text{for all $x\in \fg$}.
\]
We will only show that the squaring is well-defined. Suppose that $\tilde x - x  =i\in I_\od$ we have, 
\[
s(\tilde x)=s(x+i)=s(x)+s(i)+[x,i]=s(x) \mod (I).
\]

In the following proposition, we will show that an ordinary Lie superalgebra together with a morphism give rise to a Hom-Lie superalgebra structure on the underlying vector space.

\begin{Proposition} \label{alphamap}
Let  $( \fg, [\cdot,\cdot],s)$ be a Lie superalgebra in characteristic 2, and let $\alpha:\fg\rightarrow \fg$ be an even  superalgebra morphism. Then  $( \fg, [\cdot,\cdot]_{\alpha},s_{\alpha}, \alpha)$, where $ [\cdot,\cdot]_{\alpha}=\alpha\circ  [\cdot,\cdot]$ and $s_{\alpha}=\alpha\circ s$, is a Hom-Lie superalgebra in characteristic 2.
\end{Proposition}
\begin{proof}The first part of the proof is given in \cite{AM}. We have to check Eqns. \eqref{squaring} and \eqref{JISH}. Indeed, let $\lambda\in \mathbb{K}$ and let $x\in \fg_\od$. We have
\begin{align*}
s_{\alpha}(\lambda x)=\alpha (s(\lambda x))=\alpha(\lambda^2 s(x))=\lambda^2 \alpha(s(x))=\lambda^2 s_{\alpha}( x).
\end{align*}
On the other hand, for any $x\in \fg_\od$ and $y\in\fg$, we have
\begin{align*}
&[s_{\alpha}(x),\alpha(y)]_{\alpha}=\alpha([\alpha( s(x)),\alpha(y)])=\alpha^2([ s(x),y])=\alpha^2([x,[x,y]])\\ & =\alpha([\alpha(x),\alpha([x,y])])=[\alpha(x),[x,y]_{\alpha}]_{\alpha}.\qedhere
\end{align*} \end{proof}


More generally, let $( \fg, [\cdot,\cdot],s, \alpha)$ be a Hom-Lie superalgebra in characteristic 2 and let $\beta:\fg\rightarrow \fg$ be an even weak superalgebra morphism (the third condition of (\ref{homphialpha}) is not necessary satisfied). Then  $( \fg, [\cdot,\cdot]_{\beta}:=\beta\circ [\cdot,\cdot] ,s_{\beta}:=\beta\circ s, \alpha\circ \beta )$ is a Hom-Lie superalgebra in characteristic 2. The proof is similar to that of Proposition \ref{alphamap}.

\begin{Example}\label{examoo}
{\em Consider the ortho-orthogonal Lie superalgebra $\fg:={\mathfrak o}{\mathfrak o}_{I\Pi}^{(1)}(1|2)$ \textup{(}see \cite{BGL, LeD}\textup{)} spanned by the even vectors $h,x_2, y_2$ and the odd vectors $x_1, y_1$ with the non-zero brackets:
\[
[x_1,y_1]=[x_2,y_2]=h, \quad [h,x_1]=x_1, \quad [h,y_1]=y_1, \quad [x_2,y_1]=x_1, \quad [y_2,x_1]=y_1,
\]
and the squaring
\[
s(x_1)=x_2,\quad s(y_1)=y_2.
\]
Let us define the map $\alpha$ on the vector space underlying ${\mathfrak o}{\mathfrak o}_{I\Pi}^{(1)}(1|2)$:
\[
\begin{array} {lll}
\alpha(x_1)=\delta_1 x_1 + \delta_2 y_1, &  \alpha(y_1)=\varepsilon_1 x_1 + \varepsilon_2 y_1,&   \alpha(x_2)=\lambda_1 h+\lambda_2 x_2+ \lambda_3 y_2,\\ [2mm]
\alpha(y_2)=\beta_1 h+\beta_2 x_2+ 
\beta_3 y_2, & \alpha(h)= \gamma_1 h. & 
\end{array}\]
A direct computation shows that the map $\alpha$ is a morphism of Lie superalgebras if and only if \textup{(}where we have put for simplicity $T:=1+\delta_2 \varepsilon_1+ \delta_1 \varepsilon_2$\textup{)}:
\[
\gamma_1= (1+T)^2, \quad \beta_1=\varepsilon_1\varepsilon_2, \quad \beta_2=\varepsilon_1^2, \quad \beta_3=\varepsilon_2^2, \quad \lambda_1=\delta_1 \delta_2, \quad \lambda_2=\delta_1^2, \quad \lambda_3=\delta_2^2;
\]
together with
\begin{equation}\label{eqT}
\varepsilon_1\, T=\varepsilon_1\, T^2= \varepsilon_2\, T=\varepsilon_2\, T^2=\delta_1 \, T=\delta_1 \, T^2=\delta_2 \, T=\delta_2 \, T^2=T (1+T)=0.
\end{equation}
The only solutions to Eqns. \textup{(}\ref{eqT}\textup{)} that do  not produce the zero map are given by $T=0$.

We can, therefore, construct a Hom-Lie superalgebra by means of the map $\alpha$, depending on three parameters, as in Proposition \ref{alphamap}. So, we have
\[
\begin{array}{lll}
 \alpha(x_1)=\delta_1 x_1 + \delta_2 y_1, &  \alpha(y_1)=\varepsilon_1 x_1 + \varepsilon_2 y_1, &  \alpha(x_2)=\delta_1 \delta_2 h+\delta_1^2 x_2+\delta_2^2 y_2, \\[2mm] \alpha(y_2)=\varepsilon_1\varepsilon_2 h+\varepsilon_1^2 x_2+ \varepsilon_2^2 y_2, & \alpha(h)=  h. &
\end{array}
\]
such that $\left( \begin{array}{cc}
\varepsilon_2 & \varepsilon_1\\
\delta_2 &\delta_1
\end{array}\right)\in SL_2(\mathbb{K})$.

In particular, we have the following Hom-Lie superalgebra in characteristic 2, which we denote by  ${\mathfrak o}{\mathfrak o}_{I\Pi}^{(1)}(1|2)_\alpha$, defined by 
the brackets:
\[
[x_1,y_1]_\alpha=[x_2,y_2]_\alpha=h, \; [h,x_1]_\alpha=x_1, \; [h,y_1]_\alpha=\varepsilon x_1+y_1,\;  [x_2,y_1]_\alpha=x_1, \; [y_2,x_1]_\alpha=\varepsilon x_1+y_1,
\]
 with the corresponding squaring: 
\[
s(x_1)= x_2, \quad s(y_1)= \varepsilon h+\varepsilon^2x_2+ y_2,
\]
and the twist map  \begin{align*}
&\alpha(x_1)=x_1,\quad \alpha(y_1)=\varepsilon x_1+y_1,\quad \alpha(x_2)= x_2, \quad \alpha(y_2)=\varepsilon h+\varepsilon^2 x_2+y_2, \quad \alpha(h)=  h, 
\end{align*}
where $\varepsilon$ is a parameter in $\mathbb{K}$. We recover the Lie superalgebra ${\mathfrak o}{\mathfrak o}_{I\Pi}^{(1)}(1|2)$ for $\varepsilon=0$.
}

\end{Example} 

\subsection{The classification in low dimensions}
Let us assume here that the field $\mathbb{K}$ is infinite (for instance, algebraically closed). For the classification of Hom-Lie algebras and superalgebras  in low dimensions, see \cite{MS,GSS, GSSc,  LL,ORS1,ORS2,R,WZW}.
\subsubsection{The case $\mathrm{sdim}(\fg)=1|1$}

Assume that $\fg_\ev=\mathrm{Span}\{e\}$ and $\fg_\od=\mathrm{Span}\{f\}$. 
We set 
\begin{align*}
\alpha(e)=\lambda_1e,\quad \alpha(f)=\lambda_2f, \quad s_\fg(f)=\rho \, e,\quad
[e,e]=0,\quad [e,f]=\gamma f.
\end{align*}
It follows that $[f,f]=s(2f)-2s(f)=2s(f)=2\rho e=0.$
By straightforward computations on the conditions, one gets that the only non-trivial case is given by $\lambda_1=1$ and $\gamma=\rho \lambda_2$. Therefore, any $(1|1)$-dimensional Hom-Lie superalgebra in characteristic 2 is isomorphic to the Hom-Lie superalgebra given, with respect to basis $\{e,f\}$, by 
\begin{align*}
& \alpha(e)=e,\quad \alpha(f)=\lambda f, \quad s_\fg(f)=\rho\, e,\\
& [e,e]=0,\quad [e,f]=\rho \lambda f,\quad [f,f]=0,
\end{align*}
where $\lambda$ and $\rho$ are non-zero parameters. As the field $\mathbb{K}$ is infinite, we have a family of Hom-Lie superalgebras that depends on  parameters $\lambda$ and $\rho$.
\subsubsection{The case $\mathrm{sdim}(\fg)=1|2$ }
Assume that $\fg_\ev=\mathrm{Span}\{e\}$ and $\fg_\od=\mathrm{Span}\{f_1,f_2\}$. We define the brackets as (where $a_{i}, b_i \in \mathbb{K}$ for $i,j=1,2$):
\begin{align*}
&
[e,f_1]=a_{1}f_1+a_{2}f_2,\quad [e,f_2]=b_{1}f_1+b_{2}f_2,
\end{align*}
and finally the squaring as (where $\rho_i\in \mathbb{K}$ for $i=1,2,3$):
\[
s(f_1)=\rho_1 \, e,\quad s(f_2)=\rho_2\, e, \quad s(f_1+f_2)=\rho_3\, e.
\]
Let us consider a linear map $\alpha$ by which we will construct the Hom-structure. As $\alpha$ preserves the $\mathbb{Z}/2 \mathbb Z$-grading, and by using the Jordan decomposition we distinguish two cases:  \\

\underline{Case 1:} Suppose that $\alpha$ is given by (where $s, t_1, r_2\in \mathbb K$):
\begin{align*}
&\alpha(e)= s e,\quad \alpha(f_1)=t_1 f_1, \quad \quad \alpha(f_2)=r_2 f_2.
\end{align*}
A direct computation shows that there are only the following sub-cases:

\underline{Sub-case 1a):} We have $\rho_3=\rho_1+\rho_2,\;  \rho_1\not=0, \;  s=t_1^2, \; \rho_2\, (s+r_2^2)=0$ and $a_i=b_i=0$ for $i=1,2$. Here are the two possible cases: 
\[
\begin{array}{l} 
\rho_3=\rho_1+\rho_2,\;  \rho_1\not=0, \; \rho_2=0, \;  s=t_1^2, \; a_1=b_1=a_2=b_2=0,\; r_2 \text{ arbitrary}; \text{ or } \\[2mm]
\rho_3=\rho_1+\rho_2,\;  \rho_1, \rho_2\not=0, \;  s=t_1^2, \; t_1=r_2, \; a_1=b_1=a_2=b_2=0.
\end{array}
\]

\underline{Sub-case 1b):} We have $\rho_1=\rho_2=\rho_3=0$ together with \[
b_1(t_1 + sr_2)=0, \; b_2r_2 (1+s)=0,\; a_1 t_1 (1+s)=0,\; a_2 (r_2+st_1)=0.
\]
We can disregard this case, because it produces a Lie algebra instead of a Lie superalgebra. 

\underline{Sub-case 1c):} We have $\rho_1+\rho_2+\rho_3\not =0,\; \rho_1\not=0,$ together with 
\[ s=t_1^2= t_1 r_2,\;  \rho_2t_1^2=r_2^2\rho_2,\; a_1=b_1=a_2=b_2=0.
\]
Here are the two possible cases: 
\[
\begin{array}{l} 
\rho_1+\rho_2+\rho_3 \not =0,\;  \rho_1\not=0, \; \rho_2=0, \;  s=t_1^2=t_1 r_2,\; r_2\not=0, \; a_1=b_1=a_2=b_2=0; \text{ or } \\[2mm]
\rho_1+\rho_2+\rho_3 \not =0,\;  \rho_1, \rho_2\not=0, \;  s=t_1^2, \; t_1=r_2\not=0, \; a_1=b_1=a_2=b_2=0.
\end{array}
\]
\underline{Case 2:} Suppose that $\alpha$ is given by (where $s, t_1\in \mathbb K$):
\begin{align*}
&\alpha(e)= s e,\quad \alpha(f_1)=t_1 f_1, \quad \quad \alpha(f_2)=f_1+t_1 f_2.
\end{align*}
A direct computation shows that there are only the following sub-cases:

\underline{Subcase 2a):} We have $\rho_1, \rho_2\not=0$ but $\rho_3$ arbitrary, together with 
\[
a_1=a_2=b_1=b_2=0,\; s=t_1^2, \; \rho_1 (1+t_1) = t_1 (\rho_2+\rho_3).
\]

\underline{Subcase 2b):} We have $\rho_1, \rho_3\not=0$ but $\rho_2=0$ arbitrary, together with 
\[
a_1=a_2=b_1=b_2=0,\; s=t_1^2, \; \rho_1 (1+t_1) = t_1 \rho_3.
\]

\underline{Subcase 2c):} We have $\rho_1\not=0$ but $\rho_2=\rho_3=0$, together with 
\[
a_1=a_2=b_1=b_2=0,\; s=1, \; t_1=1.
\]

\underline{Subcase 2d):} We have $\rho_1=0,\; \rho_2\not=0$ but $\rho_3$ arbitrary, together with 
\[
a_1=a_2=b_1=b_2=0,\; s=t_1^2, \;t_1(\rho_2+\rho_3)=0.
\]

\underline{Subcase 2e):} We have $\rho_1=\rho_2=0,$\; but $\rho_3 \not=0$, together with 
\[
a_1=a_2=b_1=b_2=0,\; s=0, \; t_1=0.
\]
The tables below summarize our finding. We find it convenient to order the Hom-Lie superalgebras into two groups: (i) of type I are those for which the $\fg_0$-module structure on $\fg_\od$ is trivial; (ii) of type II are those for which the $\fg_0$-module structure on $\fg_\od$ is not trivial.
\begin{table}
\centering
\begin{tabular}{| c | l | c|}\hline
 The HLSA & The squaring $s$& The conditions \\ \hline
$A_1$ & $\begin{array}{lcl}
s(f_1) & = & \rho_1 e, \\[1mm]
s(f_2) &= &0,\\[1mm]
s(f_1+\lambda f_2)& = & \rho_1 e 
\end{array}$  & $\begin{array}{l}
\rho_1 \not =0,\; s=t_1^2,\\[1mm]
r_2 \text{ arbitrary}
\end{array}$  \\[1mm] \hline
$A_2$ & $\begin{array}{lcl}
s(f_1) & = & \rho_1 e, \\[1mm]
s(f_2) &= &\rho_2 e,\\[1mm]
s(f_1+\lambda f_2)& = & (\rho_1+ \lambda^2 \rho_2)e 
\end{array}$  & $\begin{array}{l}
\rho_1, \rho_2\not =0,\\[1mm]
s=t_1^2, r_2=t_1
\end{array}$  \\[1mm] \hline
$A_3$ & $\begin{array}{lcl}
s(f_1) & = & \rho_1 e,\\[1mm]
s(f_2) &= &0,\\[1mm]
s(f_1+\lambda f_2)& = &((1+\lambda)\rho_1+\lambda\rho_3)e 
\end{array}$  & $\begin{array}{l}
\rho_1\not =0,\;\rho_1 +\rho_3 \not =0, \\[1mm]
s=t_1^2=t_1r_2,\; r_2\not=0
\end{array}$  \\[1mm] \hline
$A_4$ & $\begin{array}{lcl}
s(f_1) & = & \rho_1 e, \\[1mm]
s(f_2) &= &\rho_2 e,\\[1mm]
s(f_1+\lambda f_2)& = &\lambda(\rho_1+(1+\lambda)\rho_2+\rho_3)e \\[1mm]
&&+\rho_1e 
\end{array}$  & $\begin{array}{l}
\rho_1, \rho_2\not =0,\;\rho_1+\rho_2 +\rho_3 \not =0, \\[1mm]
s=t_1^2, \; t_1=r_2,\; r_2\not=0
\end{array}$  \\[1mm] \hline
\end{tabular}
\caption{Type I (i.e. $[\fg_\ev, \fg_\od] =\{0\}$) with $\alpha(e)=s e,\; \alpha(f_1)=t_1 f_1,\; \alpha(f_2)=r_2 f_2$}\label{tab6}
\end{table}

\begin{table}
\centering
\begin{tabular}{| c | l | c|}\hline
 The HLSA & The squaring $s$& The conditions \\ \hline
$A_5$ & $\begin{array}{lcl}
s(f_1) & = & \rho_1 e,\\[1mm]
s(f_2) &= & \rho_2 f_2,\\[1mm]
s(f_1+\lambda f_2)& =& \lambda(\rho_1+(1+\lambda)\rho_2)e\\[1mm]
&&+(\lambda \rho_3+ \rho_1)e
\end{array}$  & $\begin{array}{l}
\rho_1, \rho_2 \not =0,\; \displaystyle \rho_3= \frac{1+t_1}{t_1}\rho_1 +\rho_2,\\[1mm]
s=t_1^2
\end{array}$  \\[1mm] \hline
$A_6$ & $\begin{array}{lcllcl}
s(f_1) & = &
\rho_1 e,  \\[1mm]
s(f_2) &= & 0,\\[1mm]
s(f_1+\lambda f_2)& = & (\lambda(\rho_1+\rho_3)+\rho_1)e  &&&
\end{array}$  & $\begin{array}{l}
\rho_1, \rho_3 \not =0,\; \displaystyle \rho_3= \frac{1+t_1}{t_1}\rho_1,\\[1mm]
s=t_1^2
\end{array}$  \\[1mm] \hline
$A_7$ & $\begin{array}{lcl}
s(f_1) & = & \rho_1 e, \\[1mm]
s(f_2) &= & 0,\\[1mm]
s(f_1+\lambda f_2)& = & \rho_1(1+\lambda)e 
\end{array}$  & $\begin{array}{l}
 \rho_1 \not = 0,\\[1mm]
s=t_1=1
\end{array}$  \\[1mm] \hline
$A_8$ & $\begin{array}{lcl}
s(f_1) & = & 0, \\[1mm] 
s(f_2) &= & \rho_2 e, \\
s(f_1+\lambda f_2)& = & \lambda (\rho_2+\rho_3+\lambda\rho_2)e 
\end{array}$  & $\begin{array}{l}
 \rho_2 \not = 0, \; \rho_3 \text{ arbitrary}\\[1mm]
s=t_1=0,
\end{array}$  \\[1mm] \hline
$A_9$ & $\begin{array}{lcl}
s(f_1) & = & 0, \\[1mm]
s(f_2) &= & \rho_2 e,\\[1mm]
s(f_1+\lambda f_2)& = & \lambda^2\rho_2 e 
\end{array}$  & $\begin{array}{l}
 \rho_2 \not = 0,\\[1mm]
s=t_1^2,\; t_1\not =0
\end{array}$  \\[1mm] \hline
$A_{10}$ & $\begin{array}{lcl}
s(f_1) & = & 0, \\[1mm]
s(f_2) &= & 0,\\[1mm]
s(f_1+\lambda f_2)& = &\lambda \rho_3 e 
\end{array}$  & $\begin{array}{l}
 \rho_3 \not = 0,\\[1mm]
s=t_1 =0
\end{array}$  \\[1mm] \hline
\end{tabular}
\caption{Type I (i.e. $[\fg_\ev, \fg_\od] =\{0\}$) with $\alpha(e)=s e,\; \alpha(f_1)=t_1 f_1,\; \alpha(f_2)=f_1+t_1 f_2$}\label{tab6}
\end{table}

\begin{table}[H]
\centering
\begin{tabular}{| c | l | l| c|}\hline
 The HLSA & The squaring $s$& $[\fg_\ev,\fg_\od]$  & The conditions \\ \hline
$B_1$ & $\begin{array}{lcl}
s(f_1) & = & 0, \\
s(f_2) &= &0,\\
s(f_1+\lambda f_2)& = & 0
\end{array}$  & $\begin{array}{lcl}
[e,f_1]& = & a_1 f_1 + a_2 f_2,\\[1mm] 
[e,f_2] & = & b_1 f_1 + b_2 f_2
\end{array}$ &$\begin{array}{l}
s=1,\\[1mm]
t_1=r_2\\[1mm]
a_1, a_2, b_1, b_2 \text{ arbitrary}
\end{array}$  \\[1mm] \hline
$B_2$ & $\begin{array}{lcl}
s(f_1) & = & 0, \\
s(f_2) &= &0,\\
s(f_1+\lambda f_2)& = & 0
\end{array}$  & $\begin{array}{lcl}
[e,f_1]& = & a_1 f_1,\\[1mm] 
[e,f_2] & = &  b_2 f_2
\end{array}$ &$\begin{array}{l}
s=1,\\[1mm]
t_1 \not =r_2\\[1mm]
a_1, b_2 \text{ arbitrary}
\end{array}$  \\[1mm] \hline
$B_3$ & $\begin{array}{lcl}
s(f_1) & = & 0, \\
s(f_2) &= &0,\\
s(f_1+\lambda f_2)& = & 0
\end{array}$  & $\begin{array}{lcl}
[e,f_1]& = & a_1 f_1 +a_2f_2,\\[1mm]
[e,f_2] & = & b_1f_1+ b_2 f_2
\end{array}$ &$\begin{array}{l}
s\not =0,1,\\[1mm]
t_1=r_2=0\\[1mm]
a_1,a_2, b_1, b_2 \text{ arbitrary}
\end{array}$  \\[1mm] \hline
$B_4$ & $\begin{array}{lcl}
s(f_1) & = & 0, \\
s(f_2) &= &0,\\
s(f_1+\lambda f_2)& = & 0
\end{array}$  & $\begin{array}{lcl}
[e,f_1]& = &0, \\[1mm]
[e,f_2] & = &  b_2 f_2
\end{array}$ &$\begin{array}{l}
s\not =0,1,\\[1mm]
t_1\not =0, r_2=0\\[1mm]
b_2 \text{ arbitrary}
\end{array}$  \\[1mm] \hline
$B_5$ & $\begin{array}{lcl}
s(f_1) & = & 0, \\
s(f_2) &= &0,\\
s(f_1+\lambda f_2)& = & 0
\end{array}$  & $\begin{array}{lcl}
[e,f_1]& = &0, \\[1mm]
[e,f_2] & = &  b_1 f_1
\end{array}$ &$\begin{array}{l}
s\not =0,1,\\[1mm]
t_1= rs_2,\\[1mm]
b_1\not=0 \text{ and arbitrary}
\end{array}$  \\[1mm] \hline
$B_6$ & $\begin{array}{lcl}
s(f_1) & = & 0, \\
s(f_2) &= &0,\\
s(f_1+ \lambda f_2)& = & 0
\end{array}$  & $\begin{array}{lcl}
[e,f_1]& = &a_2 f_2, \\[1mm]
[e,f_2] & = &  0
\end{array}$ &$\begin{array}{l}
s\not =0,1,\\[1mm]
t_1 \not = rs_2,\\[1mm]
a_2\not =0 \text{ and arbitrary}
\end{array}$  \\[1mm] \hline
$B_7$ & $\begin{array}{lcl}
s(f_1) & = & 0, \\
s(f_2) &= &0,\\
s(f_1+\lambda f_2)& = & 0
\end{array}$  & $\begin{array}{lcl}
[e,f_1]& = &a_2 f_2, \\[1mm]
[e,f_2] & = & b_2 f_2
\end{array}$ &$\begin{array}{l}
s\not =0,1,\\[1mm]
r_2=0,\\[1mm]
a_2, b_2\text{ arbitrary}
\end{array}$  \\[1mm] \hline
\end{tabular}
\caption{Type II (i.e. $[\fg_\ev, \fg_\od] \not =\{0\}$) with $\alpha(e)=s e,\; \alpha(f_1)=t_1 f_1,\; \alpha(f_2)=r_2 f_2$}\label{tab6}
\end{table}

\begin{table}[H]
\centering
\begin{tabular}{| c | l | l| c|}\hline
 The HLSA & The squaring $s$& $[\fg_\ev,\fg_\od]$  & The conditions \\ \hline
$B_8$ & $\begin{array}{lcl}
s(f_1) & = & 0, \\
s(f_2) &= &0,\\
s(f_1+\lambda f_2)& = & 0
\end{array}$  & $\begin{array}{lcl}
[e,f_1]& = & a_1 f_1,\\[1mm] 
[e,f_2] & = & b_1 f_1 + a_1 f_2
\end{array}$ &$\begin{array}{l}
s=1,\\[1mm]
t_1=0\\[1mm]
a_1, b_1 \text{ arbitrary}
\end{array}$  \\[1mm] \hline
$B_9$ & $\begin{array}{lcl}
s(f_1) & = & 0, \\
s(f_2) &= &0,\\
s(f_1+\lambda f_2)& = & 0
\end{array}$  & $\begin{array}{lcl}
[e,f_1]& = & 0,\\[1mm] 
[e,f_2] & = & b_1 f_1
\end{array}$ &$\begin{array}{l}
s\not =1,\\[1mm]
t_1=0\\[1mm]
b_1 \not =0 \text{ and arbitrary}
\end{array}$  \\[1mm] \hline
$B_{10}$ & $\begin{array}{lcl}
s(f_1) & = & 0, \\
s(f_2) &= &0,\\
s(f_1+\lambda f_2)& = & 0
\end{array}$  & $\begin{array}{lcl}
[e,f_1]& = & a_1 f_1,\\[1mm] 
[e,f_2] & = & a_1 f_2
\end{array}$ &$\begin{array}{l}
s=1,\\[1mm]
t_1 \not =0\\[1mm]
a_1 \not =0 \text{ and arbitrary}
\end{array}$  \\[1mm] \hline
\end{tabular}
\caption{Type II (i.e. $[\fg_\ev, \fg_\od] \not =\{0\}$) with $\alpha(e)=s e,\; \alpha(f_1)=t_1 f_1,\; \alpha(f_2)=f_1+t_1 f_2$}\label{tab6}
\end{table}
\section{Representations and semidirect product}
\begin{Definition}\label{defmodalpha} {\rm A \textit{representation of a Hom-Lie superalgebra} $( \fg, [\cdot,\cdot]_\fg,s_\fg, \alpha)$ is a triple $(V, [\cdot,\cdot]_V, \beta)$, where $V$ is a superspace, $\beta$ an even map  in $\fgl(V)$, and  $[\cdot,\cdot]_V$ is the action of $\fg$ on $V$  such that
\begin{equation}\label{Hom-Lie-repr}
\begin{array}{lcl}
[\alpha(x), \beta(v)]_V & = &\beta([x,v]_V)\; \text{ for any  $x\in \fg$ and $v\in V$},  \\[2mm] 
[[x,y]_\fg, \beta(v)]_V & = & [\alpha(x), [y,v]_V]_V+ [\alpha(y), [x,v]_V]_V\; \text{ for any  $x,y\in \fg$ and $v\in V$,}\\[2mm]
[s_\fg(x),\beta(v)]_V & =& [\alpha(x),[x,v]_V]_V\;\text{ for any $x\in \fg_\od$ and $v\in V$}.
\end{array}
\end{equation}
}
We say that $V$ is a $\fg$-module.
\end{Definition}
Sometimes it is more convenient to use the notation $\rho_\beta=[\cdot, \cdot]_V$ and write: 
\begin{equation}\label{Hom-Lie-reprV2}
\begin{array}{lcl}
\rho_\beta \circ \alpha(x) & = &\beta \circ \rho_\beta (x)\; \text{ for any  $x\in \fg$,} \\[2mm] 
\rho_\beta([x,y]_\fg) \circ \beta & = & \rho_\beta(\alpha(x))\rho(y) + \rho_\beta(\alpha(y))\rho(x)\; \text{ for any  $x,y\in \fg$,}\\[2mm]
\rho_\beta \circ s_\fg(x) \circ \beta& =& \rho_\beta(\alpha(x))\circ \rho_\beta(x)\;\text{ for any $x\in \fg_\od$}.
\end{array}
\end{equation}
\begin{Theorem}
Let $( \fg, [\cdot,\cdot]_\fg,s_\fg, \alpha)$ be  a Hom-Lie superalgebra and $(V, [\cdot,\cdot]_V, \beta)$ be a representation. With the above notation, we define a Hom-Lie superalgebra structure on the superspace $\fg\oplus V=(\fg_\ev+V_\ev)\oplus (\fg_\od+V_\od)$, where the bracket is defined by
$$[x+v,y+w]_{\fg\oplus V}=[x,y]_\fg+[x,w]_V+[y,v]_V  \text{ for any } x,y\in \fg  \text{ and } v,w\in V,
$$
the squaring $s_{\fg+V}:\fg_\od+V_\od\rightarrow \fg_\ev+V_\ev$ is defined by
$$s_{\fg+V}(x+v)=s_\fg(x) + [x,v]_V \text{ for any } x\in \fg_\od  \text{ and } v\in V_\od,$$ 
and the structure map $\alpha_{\fg\oplus V}:\fg\oplus V\rightarrow \fg\oplus V$ defined by
$$\alpha_{\fg\oplus V}(x+v)=\alpha(x)+\beta (v)   \text{ for any } x\in \fg  \text{ and } v\in V.
$$
The Hom-Lie superalgebra $(\fg\oplus V,[\cdot,\cdot]_{\fg\oplus V},s_{\fg+V},\alpha_{\fg\oplus V})$ is called the semidirect product of $( \fg, [\cdot,\cdot]_\fg,s_\fg, \alpha)$ by the representation  $(V, [\cdot,\cdot]_V, \beta)$.
\end{Theorem}
\begin{proof}
Checking Axioms (i) and (ii) of Definition \ref{MainDef} is a routine; we can refer to \cite{AM}. We should check the conditions relative to the squaring. Let us first check that the map $s_{\fg}\oplus V$ is indeed a squaring. For all $x+v\in \fg_\od\oplus V_\od$ and for all $\lambda \in \mathbb{K}$, we have
\[
s_{\fg\oplus V}(\lambda (x+v))=s_\fg(\lambda x)+[\lambda x, \lambda v]_V=\lambda^2 s_\fg(x)+\lambda^2[x,v]_V=\lambda^2 s_{\fg\oplus V}(x+v).
\]
Now, for all $x+v\in \fg_\od \oplus V_\od$ and for all $y+w\in \fg\oplus V$, we have
\begin{align*}
& [s_{\fg \oplus V}(x+v),\alpha_{\fg\oplus V}(y+w)]_{\fg \oplus V} =[s_{\fg}(x)+[x,v]_V,\alpha(y)+ \beta(w)]_{\fg \oplus V} \\[2mm]& =[s_{\fg}(x),\alpha(y)]_\fg+[s_{\fg}(x),\beta(w)]_V+
[\alpha(y), [x,v]_V]_V\\[2mm]
&=[\alpha(x),[x,y]_\fg]_\fg+ [\alpha(x), [x,w]_V]_V+[\alpha(y), [x,v]_V]_V.
\end{align*}
On the other hand, 
\begin{align*}
&
[\alpha_{\fg\oplus V}(x+v), [x+v, y+w]_{\fg\oplus V}]_{\fg\oplus V}=[\alpha(x)+ \beta(v), [x,y]_\fg+[x,w]_V+[y,v]_V]_{\fg\oplus V}\\[2mm]
&=[\alpha(x), [x,y]_\fg]_\fg+ [\alpha(x), [x,w]_V+[y,v]_V]_V+[[x,y]_\fg, \beta(v)]_V\\[2mm]
& = [\alpha(x), [x,y]_\fg]_\fg+ [\alpha(x), [x,w]_V+[y,v]_V]_V+[\alpha(x), [y,v]_V]_V+ [\alpha(y), [x,v]_V]_V\\[2mm]
& = [\alpha(x), [x,y]_\fg]_\fg+ [\alpha(x), [x,w]_V]_V+ [\alpha(y), [x,v]_V]_V.
\end{align*}
Therefore, Eq. \eqref{JISH} is satisfied. Now, 
\begin{align*}
&\alpha_{\fg\oplus V}(s_{\fg\oplus V}(x+v))= \alpha_{\fg\oplus V}(s_{\fg}(x)+[x,v]_V)=\alpha(s_\fg(x))+\beta([x,v]_V)\\[2mm]
& =\alpha(s_\fg(x))+[\alpha(x),\beta(v)]_V= s_\fg(\alpha(x))+[\alpha(x),\beta(v)]_V= s_{\fg\oplus V}(\alpha(x)+\beta(v))\\[2mm]
&=s_{\fg\oplus V}(\alpha_{\fg\oplus V}(x+v)).
\end{align*}
Therefore, Eq. \eqref{multsq} is satisfied.
\end{proof}
In the following proposition, we show how to  twist a Lie superalgebra and its representation into a Hom-Lie superalgebra together with a representation in characteristic 2.

\begin{Proposition} \label{betamap}
Let  $( \fg, [\cdot,\cdot]_\fg,s_\fg)$ be a Lie superalgebra and $(V,\rho ) $ a representation.  Let $\alpha:\fg\rightarrow \fg$ be an even  superalgebra morphism and $\beta \in\fgl(V)$ be a linear map such that $ \rho(\alpha (x))\circ \beta =\beta \circ \rho(x)$.  Then  $( \fg, [\cdot,\cdot]_{\fg,\alpha},s_{\fg,\alpha}, \alpha)$, where $ [\cdot,\cdot]_{\fg,\alpha}=\alpha\circ  [\cdot,\cdot]_{\fg}$ and $s_{\fg,\alpha}=\alpha\circ s_{\fg}$, is a Hom-Lie superalgebra and $(V,\rho_\beta,\beta)$, where $\rho_\beta =\beta\circ \rho$,  is a representation.
\end{Proposition}
\begin{proof} We have already proved in Proposition \ref{alphamap} that  $( \fg, [\cdot,\cdot]_{\fg,\alpha},s_{\fg,\alpha}, \alpha)$ is a Hom-Lie superalgebra. Let us check that $(V,\rho_\beta,\beta)$ is a representation with respect to $( \fg, [\cdot,\cdot]_{\fg,\alpha},s_{\fg,\alpha}, \alpha)$. Indeed, the first condition is provided by the hypothesis while the second and the third ones are straightforward.  Let us check the last one. For any $x\in \fg_\od$ and $v\in V$, we have 
\begin{align*}
[s_{\fg,\alpha}( x),\beta(v)]_{V,\beta}=\beta([\alpha (s_{\fg}( x)),\beta(v)]_V)=\beta^2([s_\fg(x),v]_V),
\end{align*}
and
\begin{align*}
[\alpha(x),[x,v]_{V,\beta}]_{V,\beta}=\beta[\alpha(x),\beta( [x,v]_{V}]_{V})=\beta^2([x, [x,v]_{V}]_{V}).
\end{align*} 
The equality follows from the fact that $[s_\fg(x),v]_V=[x, [x,v]_{V}]_{V}$.
\end{proof}
\begin{Example} {\em  The classification of irreducible modules over $\mathfrak{oo}_{I\Pi}^{(1)}(1|2)$ having highest weight vectors has been carried out in \cite{BGKL}. We will borrow here the simplest example. Consider the Hom-Lie superalgebra $\mathfrak{oo}_{I\Pi}^{(1)}(1|2)$ with the twist $\alpha$ given as in Example \textup{(}\ref{examoo}\text{)}. We consider the $\mathfrak{oo}_{I\Pi}^{(1)}(1|2)$-module $M$ with basis: \textup{(}even $|$ odd\textup{)}
\[
m_1, m_3 \quad | \quad m_2.
\]
The vector $m_1$ is a highest weight  vector with $\text{weight }(m_1)=(1)$. The map $\beta$ is given as follows:
\[
\beta(m_1)=\delta_1 m_1+ \delta_2 m_3, \quad \beta(m_3)=\varepsilon_1 m_1+\varepsilon_2m_3, \quad \beta(m_2)= m_2,
\]
where the coefficients $\delta_1, \delta_2, \varepsilon_1, \varepsilon_2$ are given as in Example \ref{examoo}.
}
\end{Example}
Here we will introduce another point of view concerning the representations of Hom-Lie superalgebras in characteristic 2, inspired by \cite{SX}.

Let $V=V_\ev\oplus V_\od$ be a vector superspace, and let $\beta \in GL(V)$ be even map. We will define a bracket on $\fgl(V)$ as well as a product as follows: (where $\beta^{-1}$ is the inverse of $\beta$):
\begin{eqnarray}
\label{braglV}[f,g]_{\fgl(V)} & := & \beta \circ f \circ  \beta^{-1} g \circ \beta^{-1}+  \beta \circ g \circ  \beta^{-1} f \circ \beta^{-1} \quad \text{ for all $f,g\in \fgl(V)$}, \\[1mm]
\label{sqglV} s_{\fgl(V)}(f) &: =& \beta \circ f \circ  \beta^{-1} f \circ \beta^{-1} \quad \text{ for all $f\in \fgl(V)_\od$}.
\end{eqnarray}
Obviously, $s_{\fgl(V)}(\lambda f)=\lambda^2s_{\fgl(V)}(f)$ for all $\lambda\in {\mathbb K}$ and for all $f\in \fgl(V)_\od$. Now, the map 
\[
(f,g)\mapsto s_{\fgl(V)}(f+g)+s_{\fgl(V)}(f)+s_{\fgl(V)}(g)=\beta \circ f \circ  \beta^{-1} g \circ \beta^{-1}+\beta \circ g \circ  \beta^{-1} f \circ \beta^{-1}
\]
is obviously bilinear on $\fgl(V)_\od$ as well.\\

Denote by $\mathrm{Ad}_\beta : \fgl(V ) \rightarrow \fgl(V )$ the adjoint action on $\fgl(V )$, i.e. $\mathrm{Ad}_\beta(f)=\beta \circ f\circ \beta^{-1}$.
\begin{Proposition}
The brackets and the squaring defined in \rm{Eqns}. \textup{(}\ref{braglV}\textup{)} and  \textup{(}\ref{sqglV}\textup{)} make $(\fgl(V), [\cdot,\cdot]_{\fgl (V)}, s_{\fgl(V)}, \mathrm{Ad}_{\beta})$ a Hom-Lie superalgebra in characteristic 2.
\end{Proposition}
\begin{proof} The map $\mathrm{Ad}_{\beta}$ is invertible with inverse $\mathrm{Ad}_{\beta^{-1}}$. Let us check the multiplicativity conditions:
\begin{align*}
& [\mathrm{Ad}_\beta(f),\mathrm{Ad}_\beta(g) ]_{\fgl(V)}=[\beta \circ f\circ \beta^{-1},\beta \circ g \circ \beta^{-1} ]_{\fgl(V)}\\[2mm]
&=\beta \circ (\beta \circ f\circ \beta^{-1}) \circ  \beta^{-1} (\beta \circ g\circ \beta^{-1}) \circ \beta^{-1}+  \beta \circ (\beta \circ g\circ \beta^{-1}) \circ  \beta^{-1} (\beta \circ f\circ \beta^{-1}) \circ \beta^{-1} \\[2mm]
&= \beta \circ (\beta \circ f\circ \beta^{-1} \circ g\circ \beta^{-1}) \circ \beta^{-1}+  \beta \circ (\beta \circ g\circ \beta^{-1}  \circ f\circ \beta^{-1}) \circ \beta^{-1} \\[2mm]
&= \mathrm{Ad}_\beta([f,g]_{\fgl(V)}).
\end{align*}
Similarly, 
\begin{align*}
&s_{\fgl(V)}(\mathrm{Ad}_\beta(f))=s_{\fgl(V)}(\beta \circ f \circ \beta^{-1})=\beta \circ (\beta \circ f \circ \beta^{-1})\circ  \beta^{-1} \circ (\beta \circ f \circ \beta^{-1}) \circ \beta^{-1}\\[2mm]
&=\beta \circ (\beta \circ f \circ \beta^{-1} \circ f \circ \beta^{-1}) \circ \beta^{-1}= \mathrm{Ad}_\beta(s_{\fgl(V)}).
\end{align*}
For the Jacobi identity, let us just deal with the squaring. The LHS of the JI reads (for all $f\in \fgl(V )_\od$ and for all $g\in \fgl(V )$)
\[
\begin{array}{lcl}
[s_{\fgl(V )}(f),\mathrm{Ad}_\beta(g)]_{\fgl(V )}&=&\beta \circ s_{\fgl(V )}(f) \circ  \beta^{-1} \circ \beta\circ g \circ \beta^{-1} \circ \beta^{-1}+ \\[2mm]
&&+\beta \circ \beta\circ g \circ \beta^{-1} \circ  \beta^{-1} \circ s_{\fgl(V )}(f) \circ \beta^{-1} \\[2mm]
& =& \beta^2 \circ ( f \circ  \beta^{-1} \circ f \circ \beta^{-1} \circ  g + g \circ  \beta^{-1} \circ f \circ  \beta^{-1} \circ  f ) \circ \beta^{-2}.
\end{array}
\]
While the RHS reads
\[
\begin{array}{lcl}
[\mathrm{Ad}_\beta(f),[f,g]_{\fgl(V )}]_{\fgl(V )}&=&  \beta^2 \circ f \circ \beta^{-2} \circ  [f,g]_{\fgl(V )} \circ \beta^{-1}+  \beta \circ [f,g]_{\fgl(V )} \circ f \circ \beta^{-1} \circ \beta^{-1}\\[2mm]
&=&  \beta^2 \circ f \circ  \beta^{-2} \circ  (\beta \circ f \circ  \beta^{-1} g \circ \beta^{-1}+  \beta \circ g \circ  \beta^{-1} f \circ \beta^{-1} ) \circ \beta^{-1}\\[2mm]
&&+  \beta \circ (\beta \circ f \circ  \beta^{-1} g \circ \beta^{-1}+\beta \circ g \circ  \beta^{-1} f \circ \beta^{-1} ) \circ f \circ \beta^{-2}\\[2mm]
&=& \beta^2 \circ ( f \circ  \beta^{-1} \circ f \circ \beta^{-1} \circ  g + g \circ  \beta^{-1} \circ f \circ  \beta^{-1} \circ  f ) \circ \beta^{-2}.\qedhere
\end{array}
\]\end{proof}
\begin{Theorem}
Let $(\fg, [\cdot,\cdot]_\fg, s_\fg, \alpha)$ be a Hom-Lie superalgebra in characteristic 2. Let $V$ be a vector superspace, and let $\beta\in GL(V)$ be even. Then, the map $\rho_\beta:\fg\rightarrow \fgl(V)$ is a representation of $(\fg, [\cdot,\cdot]_\fg, s_\fg, \alpha)$ on $V$ with respect to $\beta$ if and only if the map $\rho_\beta: (\fg, [\cdot,\cdot]_\fg, s_\fg, \alpha)
 \rightarrow (\fgl(V), [\cdot, \cdot]_{\fgl (V)},s_{\fgl(V)}, \mathrm{Ad}_{\beta})$ is a morphism of Hom-Lie superalgebras.\end{Theorem}
 \begin{proof}
 Let us only proof one direction. Suppose that $\rho_\beta:\fg\rightarrow \fgl(V)$ is a representation of  $(\fg, [\cdot,\cdot]_\fg, s_\fg, \alpha)$ on $V$ with respect to $\beta$. Since $\rho_\beta(\alpha(x))\circ \beta=\beta \circ \rho(x)$, for all $f\in \fg$, it follows that 
 \[
 \rho_\beta (x)\circ \alpha= \beta \circ \rho(x)\circ \beta^{-1}= \mathrm{Ad}_\beta \circ \rho_\beta(x).
 \]
 Now, 
 \[
 \begin{array}{lcl}
     \rho_\beta([x,y]_\fg) & = &  \rho_\beta(\alpha(x)) \circ \rho(y) \circ  \beta^{-1}+ \rho_\beta(\alpha(y))  \circ \rho(x) \circ \beta^{-1} \\[2mm]
      &  = & \rho_\beta(\alpha(x)) \circ \beta \circ \beta^{-1} \circ \rho_\beta(y) \circ  \beta^{-1}+ \rho_\beta(\alpha(y)) \circ \beta \circ \beta^{-1} \circ \rho_\beta(x) \circ \beta^{-1} \\[2mm]
       &  = & \beta \circ \rho_\beta(x) \circ \beta^{-1} \circ \rho_\beta(y) \circ  \beta^{-1}+ \beta\circ  \rho_\beta(y) \circ \beta^{-1} \circ \rho_\beta(x) \circ \beta^{-1} \\[2mm]
       & = & [\rho_\beta(x), \rho_\beta(y)]_{\fgl(V)}
 \end{array}
 \]
 For the squaring, we have
 \[
 \begin{array}{lcl}
 \rho_\beta(s_\fg(x)) & = & \rho_\beta(\alpha(x)) \circ \rho_\beta(x)\circ \beta^{-1}\\[2mm]
 & = & \beta \circ \rho_\beta(x) \circ \beta^{-1} \circ \rho_\beta(x) \circ \beta^{-1}\\[2mm]
 &= & s_{\fgl(V)} ( \rho_\beta(x)).
 \end{array}
 \]
 It follows that $\rho_\beta$ is a homomorphism of Hom-Lie superalgebras in characteristic 2.
 \end{proof}
 \begin{Corollary}
 Let $(\fg, [\cdot, \cdot]_\fg, s_\fg, \alpha)$ be a Hom-Lie superalgebra in characteristic 2. Then, the adjoint representation $\mathrm{ad} : \fg \rightarrow \fgl(\fg)$, which is defined by $\mathrm{ad}_x(y) = [x, y]_\fg$, is a morphism from $(\fg, [\cdot, \cdot]_\fg, \alpha)$ to $(\fgl(\fg), [\cdot, \cdot]_{\fgl(\fg)},s_{\fgl(\fg)}, \mathrm{Ad}_\alpha)$.
 \end{Corollary}
\section{$\alpha^k$-Derivations}
Let  $( \fg, [\cdot,\cdot],s, \alpha)$ be a Hom-Lie superalgebra in characteristic 2. We denote by $\alpha^k$ the $k$-times  composition of $\alpha$,  where $\alpha^0$ is the identity map. We will be needing the following linear map 
\begin{equation} \label{adalpha}
\ad_{\alpha^{s,k}}(x): y \mapsto [\alpha^s(x), \alpha^k(y)].  
\end{equation}

\begin{Definition}
{\rm A linear map $D:\fg\rightarrow \fg$ is called an \textit{$\alpha^k$-derivation} of the Hom-Lie superalgebra $\fg$ if
\begin{eqnarray}
\label{Der0H} D\circ \alpha& = & \alpha\circ D, \text{ namely $D$ and $\alpha$ commutes.}\\[2mm]
\label{Der1H} D([x,y])&=&[D(x),\alpha^k(y)]+[\alpha^k(x),D(y)]\quad \text{for any $x\in \fg_\ev$ and $y\in \fg$ }.\\[2mm]
\label{Der2H} D(s(x))&=&[D(x),\alpha^k(x)]\quad \text{for any $x\in \fg_\od$}.
\end{eqnarray}
}
\end{Definition}

\begin{Remark}{\rm 
Notice that condition (\ref{Der2H}) implies condition (\ref{Der1H}) if $x,y\in \fg_\od$. }
\end{Remark}
Let us give an example. Let $x\in \fg$ such that $\alpha(x)=x$. The linear map $\ad_{\alpha^{0,k} }(x): y \mapsto [x, \alpha^k(y)]$ (see Eq. (\ref{adalpha})) is an $\alpha^k$-derivation. Let us just check the condition related to the squaring. Indeed,
\[ 
\ad_{\alpha^{0,k}}(x)(s(y))=[x,\alpha^k(s(y))] = [x,s(\alpha^k(y))] = [[x,\alpha^k(y)],\alpha^k(y)] = [\ad_{\alpha^{0,k}}(x)(y),\alpha^k(y)]
\]
Let us denote the space of $\alpha^k$-derivations by $\fder^{\alpha}(\fg)$. We have the following proposition.
\begin{Proposition}
The space $\fder^{\alpha}(\fg)$ can be endowed with a Lie superalgebra structure in characteristic 2. The bracket is the usual commutator, and the squaring is given by
\[
s_{\fder^{\alpha}(\fg)}(D):=D^2 \quad \text{ for all $D\in \fder^{\alpha}_\od(\fg)$.}
\]
\end{Proposition}
\begin{proof}
As we did before, we only prove the requirements when the squaring is involved. Let us first show that $D^2$ is an $\alpha^{2k}$-derivation. Checking the bracket is a routine. For the squaring, we have (for all $x\in \fg_\od$):
\begin{align*}
& D^2(s_\fg(x))=D([D(x), \alpha^k(x)]_\fg)=[D^2(x), \alpha^k(\alpha^k(x))]_\fg+ [\alpha^k(D(x)), D(\alpha^k(x))]_\fg\\[2mm]
&=[D^2(x), \alpha^{2k}(x)]_\fg+ [\alpha^k(D(x)), \alpha^k(D(x))]_\fg=[D^2(x), \alpha^{2k}(x)]_\fg.
\end{align*}
Before we proceed with the proof, let us re-denote  the space $\fder^{\alpha}(\fg)$ by $\mathfrak h$ for simplicity.

Now, for all $D\in {\mathfrak h}_\od$ and for all $E\in {\mathfrak h}_\od$, we have (for all $x\in \fg$):
\[
[s_{\mathfrak h}(D), E]_{\mathfrak h}(x)=[D^2, E]_{\mathfrak h}(x)=D^2\circ E(x)+ E \circ D^2 (x).
\]
On the other hand,
\begin{align*}
&[D,[D,E]_{\mathfrak h}]_{\mathfrak h}(x)=[D, D\circ E+E\circ D]_{\mathfrak h}(x) \\[2mm]
& =D \circ ( D\circ E+E\circ D)(x)+  (D\circ E+E\circ D)\circ D(x)\\[2mm]
&=D^2\circ E(x)+ E \circ D^2 (x).
\end{align*}
Therefore, $[s_{\mathfrak h}(D), E]_{\mathfrak h}=[D,[D,E]_{\mathfrak h}]_{\mathfrak h}$.\end{proof}
The space $\fder^{\alpha}(\fg)$ is actually graded as $\fder^{\alpha}(\fg)=\oplus \fder^{\alpha}_k(\fg)$ where $\fder^{\alpha}_k(\fg)$ is the space of $\alpha^k$-derivations where $k$ is fixed. Indeed, we have 
\[
[\fder^{\alpha}_k(\fg),\fder^{\alpha}_l(\fg)] \subseteq \fder^{\alpha}_{k+l}(\fg)\quad \text{ and }\quad  s(\fder^{\alpha}_k(\fg)_\od)\subseteq \fder^{\alpha}_{2k}(\fg).
\]
\begin{Example} {\rm We will describe all $\alpha^k$-derivations of the Hom-Lie superalgebra ${\mathfrak o}{\mathfrak o}_{I\Pi}^{(1)}(1|2)_\alpha$ introduced in Example \ref{examoo}. First, observe that 
$$
\alpha^{2k}=\alpha^0=\mathrm{Id}, \quad \alpha^{2k+1}=\alpha, \quad \text{for all $k\geq 0$}.
$$
The case of $\alpha^0$-derivations: 
\[
\begin{array}{llcl}
\text{(Even)} & D_1^0 & = & h_1\otimes y_2^*+x_1\otimes y_1^*,\\[2mm]
\text{(Even)} & D_2^0 & = & x_1\otimes x_1^*+y_1\otimes y_1^*,\\[2mm]
\text{(Odd)} & D_3^0 & = & x_1\otimes h^*_1+h_1\otimes y_1^*+y_1\otimes y_2^*.
\end{array}
\]
The case of $\alpha$-derivations:
\[
\begin{array}{llcl}
\text{(Even)} & D_1^1& =& h_1\otimes y_2^*+x_1\otimes y_1^*,\\[2mm]
   \text{(Even)} & D_2^1& =&\epsilon \,
   x_1\otimes y_1^*+x_1\otimes x_1^*+y_1\otimes y_1^*,\\[2mm]
\text{(Odd)} & D_3^1& =& \epsilon \,
   x_1\otimes y_2^*+x_1\otimes h_1^*+h_1\otimes y_1^*+y_1\otimes y_2^*.
\end{array}
\]
}
\end{Example}

\section{$p$-structures and queerification of Hom-Lie algebras in characteristic 2}
We will first introduce the concept of $p$-structures on Hom-Lie algebras. In the case of Lie algebras, the definition is due to Jacobson \cite{J}. Roughly speaking, one requires the existence of an endomorphism on the modular
Lie algebra that resembles the pth power mapping $x \mapsto  x^p$ in associative algebras. In the case of Hom-Lie algebra, there is a definition proposed in \cite{GC} but it turns out that this definition is not appropriate to queerify a restricted Hom-Lie algebras in characteristic two, as done in \cite{BLLSq} in the case of ordinary restricted Lie algebras. Here, we will give an alternative definition and justify the construction. 
\begin{Definition} {\rm Let $\fg$ be a Hom-Lie algebra in characteristic $p$ with a twist $\alpha$. A mapping $[p]_\alpha : \fg \rightarrow \fg,\; a \mapsto a^{[p]_\alpha}$ is called a $p$-structure of
$\fg$ and $\fg$ is said to be restricted if
\begin{enumerate}
    \item[\textup{(}R1\textup{)}] $\ad(x^{[p]_\alpha}) \circ \alpha^{p-1} = \ad (\alpha^{p-1}(x)) \circ \ad(\alpha^{p-2}(x))\circ \cdots \circ \ad(x)$ for all $x \in \fg$;\\[2mm]
    \item[\textup{(}R2\textup{)}] $(\lambda x)^{[p]_\alpha} = \lambda^p x^{[p]_\alpha}$ for all $x \in  \fg$ and for all  $\lambda \in \mathbb{K}$;\\[2mm]
    \item[\textup{(}R3\textup{)}] $(x + y)^{[p]_\alpha} = x^{[p]_\alpha} + y^{[p]_\alpha}+ \displaystyle \sum_{1\leq i \leq p -1} s_i (x, y)$, where the $s_i(x,y)$ can be obtained from 
    \[
    \ad({\alpha^{p-2}}(\lambda x+y))\circ\ad({\alpha^{p-3}}(\lambda x+y)) \circ \cdots \circ \ad (\lambda x+y)(x)= \sum_{1\leq i \leq p-1} i s_i(x,y) \lambda^{i-1}.
    \]
\end{enumerate}
}
\end{Definition}
Let us exhibit this $p$-structure in the case where $p=2$. The conditions (R2) and (R3) read, respectively, as 
\[
[x^{[2]_\alpha}, \alpha(y)]= [\alpha(x), [x,y]]\; \text{ and } \;  (x+y)^{[2]_\alpha}=x^{[2]_\alpha}+ y^{[2]_\alpha}+ [x,y].
\]
\begin{Proposition}
Twisting with a morphism $\alpha$ an ordinary Lie algebra with a $p$-structure gives rise to a Hom-Lie algebra with a $p$-structure. More precisely, given an ordinary Lie algebra $(\fg, [\cdot, \cdot])$ and a Lie algebra morphism $\alpha$. Then $(\fg, [\cdot, \cdot]_{\alpha}, \alpha )$, where $[\cdot, \cdot]_{\alpha}:=\alpha \circ [\cdot ,\cdot]$,  is a Hom-Lie algebra with a $p$-structure given by 
\[
x^{[p]_\alpha}:= \alpha^{p-1}(x^{[p]}).
\]
\end{Proposition}
\begin{proof}
It has been shown in \cite{Y} that if $(\fg, [\cdot, \cdot])$ is an ordinary Lie algebra, then $(\fg, [\cdot, \cdot]_{\alpha})$ where $[\cdot, \cdot]_{\alpha}:=\alpha \circ [\cdot ,\cdot]$, is a Hom-Lie algebra. Now, let us show that the map $[p]_\alpha$ defines a $p$-structure on the Hom-Lie algebra $(\fg, [\cdot, \cdot]_\fg)$. Indeed, let us check Axiom (R1). The LHS reads
\[
\ad(x^{[p]_\alpha}) \circ \alpha^{p-1}(y)=[x^{[p]_\alpha}, \alpha^{p-1}(y)]_\alpha=\alpha( [\alpha^{p-1} ( x^{[p]}),\alpha^{p-1}(y)])=\alpha^{p}([x^{[p]}, y]).
\]
The RHS reads 
\begin{align*}
&\ad (\alpha^{p-1}(x)) \circ \ad(\alpha^{p-2}(x))\circ \cdots \circ \ad(x)(y)=[\alpha^{p-1}(x), [\alpha^{p-2}(x), \ldots,[x,y]_\alpha]_\alpha\\[2mm]
&=\alpha ([\alpha^{p-1}(x), \alpha( [\alpha^{p-2}(x),[ \ldots,\alpha( [x,y]))])=\alpha^p ([x, [x,\ldots,[x,y]])= \alpha^p([x^{[p]}, y]).
\end{align*}
Axiom (R2) is obviously satisfied. Let us check Axiom (R3). Indeed, 
\begin{align*}
&(x+y)^{[p]_\alpha}=\alpha^{p-1}((x+y)^{[p]})=\alpha^{p-1}\left (x^{[p]}+y^{[p]} +\sum_{1 \leq  i\leq p-2} s_i(x,y) \right )\\[2mm]
&= x^{[p]_\alpha}+ y^{[p]_\alpha}+\left ( \sum_{1 \leq i\leq p-2} \alpha^{p-1} (s_i(x,y)) \right ).
\end{align*}
Now, 
 \begin{align*}
&    \ad({\alpha^{p-2}}(\lambda x+y))\circ\ad({\alpha^{p-3}}(\lambda x+y)) \circ \cdots \circ \ad (\lambda x+y)(x)\\[2mm]
&=[\alpha^{p-2}(\lambda x+y),[\alpha^{p-3}(\lambda x+y), \ldots, [\lambda x+y,x]_\alpha]_\alpha  \\[2mm]
&=\alpha ([\alpha^{p-2}(\lambda x+y),\alpha([\alpha^{p-3}(\lambda x+y),[ \ldots, \alpha ([\lambda x+y,x]))])  \\[2mm]
&=\alpha^{p-1} ([\lambda x+y,[\lambda x+y,[ \ldots,  [\lambda x+y,x]]) \\[2mm]
&= \alpha^{p-1}\left (\sum_{1\leq i \leq p-1} i s_i(x,y) \lambda^{i-1} \right )= \sum_{1\leq i \leq p-1} i \alpha^{p-1}(s_i(x,y)) \lambda^{i-1}.
    \end{align*}
The proof is now complete.
\end{proof}

\begin{Proposition}
Let $\fg$ be a restricted Hom-Lie algebra in characteristic 2 with a twist map $\alpha$. On the superspace $\fh:=\fg\oplus \Pi (\fg)$ there exists a Hom-Lie superalgebra structure defined as follows \textup{(}for all $x,y \in \fg$\textup{)}:
\[
[x,y]_\fh:=[x,y]_\fg, \quad [\Pi(x),y]_\fh:=\Pi ([x,y]_\fg), \quad s_\fh(\Pi(x))=x^{[2]_\alpha}.
\]
\end{Proposition}
\begin{proof} Let us check that the map $s_\fh$ is indeed a squaring on $\fh$. The condition $s_\fh(\lambda \Pi(x))=\lambda^2 s_\fh(\Pi(x))$, for all $\lambda \in \mathbb{K}$ and for all $x\in \fg$, is an immediate consequence of condition (R2). Moreover, the map 
\[
(\Pi(x), \Pi(y)\mapsto s_\fh(\Pi(x)+\Pi(y))+s_\fh(\Pi(x))+ s_\fh(\Pi(y))=(x+y)^{[2]_\alpha}+ x^{[2]_\alpha}+ y^{[2]_\alpha}=[x,y]_\fg
\]
is obviously bilinear because it coincides with the Lie bracket on $\fg$.

Let us check the Jacobi identity  involving the squaring. Indeed, for all $y\in \fh_\ev$ and for all $\Pi(x)\in \fh_\od$, we have 
\begin{align*}
&[s_\fh(\Pi(x)),\alpha(y)]_\fh=[x^{[2]_\alpha}, \alpha(y)]_\fh=[x^{[2]_\alpha}, \alpha(y)]_\fg=[\alpha(x), [x,y]_\fg]_\fg.
\end{align*}
On the other hand
\begin{align*}
&[\alpha(\Pi(x)), [\Pi(x),y]_\fh]_\fh=[\Pi(\alpha(x)), \Pi([x,y]_\fg)]_\fh=\Pi([\Pi(\alpha(x)), [x,y]_\fg]_\fh)=\Pi^2([\alpha(x), [x,y]_\fg]_\fh)\\[2mm]
&=[\alpha(x), [x,y]_\fg]_\fg.
\end{align*}
For all $\Pi(y)\in \fh_\od$ and for all $\Pi(x)\in \fh_\od$,  we have 
\begin{align*}
&[s_\fh(\Pi(x)),\alpha(\Pi(y))]_\fh=[x^{[2]_\alpha}, \alpha(\Pi(y))]_\fh=\Pi([x^{[2]_\alpha}, \alpha(y)]_\fg)=\Pi([\alpha(x), [x,y]_\fg]_\fg).
\end{align*}
On the other hand
\begin{align*}
&[\alpha(\Pi(x)), [\Pi(x),\Pi(y)]_\fh]_\fh=[\Pi(\alpha(x)), s_\fh(\Pi(x)+\Pi(y))+s_\fh(\Pi(x))+s_\fh(\Pi(x)]_\fh\\[2mm]
&=[\Pi(\alpha(x)), (x+y)^{[2]_\alpha}+x^{[2]_\alpha}+y^{[2]_\alpha}]_\fh=[\Pi(\alpha(x)), [x,y]_\fg]_\fh=\Pi([\alpha(x), [x,y]_\fg]_\fg).\qedhere
\end{align*}\end{proof}

\begin{Proposition}
Let $\fg$ be a restricted Lie algebra in characteristic 2 and  $\fh:=\fg\oplus \Pi (\fg)$ be its  queerification, see \cite{BLLSq},   defined as follows \textup{(}for all $x,y \in \fg$\textup{)}:
\[
[x,y]_\fh:=[x,y]_\fg, \quad [\Pi(x),y]_\fh:=\Pi ([x,y]_\fg), \quad s_\fh(\Pi(x))=x^{[2]}.
\]
Let $\alpha:\fg\rightarrow \fg $ be a Lie algebra morphism. Let us extend it to $\tilde \alpha$ on $\fh$ by declaring $\alpha (\Pi (x)):=\Pi (\alpha(x))$ for all $x\in \fg$.   Then twisting the Lie superalgebra $\fh$ along $\tilde{\alpha}$ is exactly the  queerification of the Hom-Lie algebra $\fg_\alpha$ obtained by twisting $\fg$ along $\alpha$. Namely,
$$
\fh_{\tilde{\alpha}}=(\fg \oplus \Pi (\fg))_{\tilde{\alpha}}=\fg_\alpha \oplus \Pi (\fg_\alpha).
$$
\end{Proposition}
\begin{proof} Let $x,y\in \fg$. We have 
\[
[x,y]_{\fh_{\tilde \alpha}}=\tilde \alpha( [x,y]_\fh)=\alpha([x,y]_\fg).
\]
On the other hand, 
\[
[x,y]_{\fg_\alpha\oplus \Pi(\fg_\alpha)}=[x,y]_{\fg_\alpha}= \alpha([x,y]_\fg).
\]
Similarly, one can easily prove that
\[
[\Pi(x),y]_{\fh_{\tilde \alpha}}=[\Pi(x),y]_{\fg_\alpha \oplus \Pi(\fg_\alpha)}.
\]
Let us only prove that their squarings coincide. Indeed, for all $x\in \fg$ we have
\[
s_{\fh_{\tilde \alpha}}(\Pi(x))= \alpha\circ s_{\fh}(\Pi(x))=\alpha (x^{[2]}).
\]
On the other hand, 
\[
s_{\fg_\alpha\oplus \Pi (\fg_\alpha)}(\Pi(x))=x^{[2]_\alpha}=\alpha (x^{[2]}).\qedhere
\]
\end{proof}

\section{Cohomology and Deformations of finite dimensional  Hom-Lie superalgebras}
\subsection{Cohomology of ordinary Lie superalgebras in characteristic 2} In this section we  define a cohomology theory of Lie superalgebras in characteristic 2. The first instances can be found in  \cite{BGLL1}. Let $\fg$ be a Lie superalgebra in characteristic 2 and $M$ be a $\fg$-module.
Let us introduce a map 


\begin{equation}\label{mapp}
\mathfrak p:\fg_\od \times \wedge ^n\fg \rightarrow M,
\end{equation}
with the following properties:

(i) $\mathfrak{p}(\lambda x,z)=\lambda^2 \mathfrak{p}(x,z)$ for all $x\in \fg_\od$, for all $z\in \wedge^n\fg$ and for all $\lambda \in \mathbb K$;

(ii) For all $x\in \fg_\od$, the map $z\mapsto {\mathfrak p}(x,z)$ is multi-linear.

For $n=0$, the map ${\mathfrak p}$ should be understood as a quadratic form on $\fg_\od$ with values in~$M$. 

We are now ready to define the space of cochains on $\fg$ with values in $M$. We set ($n>1$)
\begin{equation}\label{xcochains}
\begin{array}{lcl}
XC^{-1}(\fg;M)&:=& \{0\},\\[2mm]
XC^0(\fg;M)&:=& M,\\[2mm]
XC^1(\fg;M)&:=&\{c \;| \text{ where $c:\fg \rightarrow M$ is linear} \}, \\[2mm]
 XC^n(\fg;M)&:=& \{ (c,\mathfrak{p}) \; | \text{ where } c:\wedge^n \fg \rightarrow M  \text{ is a multi-linear map and } 
      \mathfrak p:\fg_\od \times  \wedge^{n-2} \fg \rightarrow M  \\[1mm]
&& \text{ is a map as in \eqref{mapp} }
 \text{ such that }
\mathfrak{p}(x+y,z)+\mathfrak{p}(x,z)+\mathfrak{p}(y,z)=c(x,y,z)   \\[1mm]
 && \text{ for all }  x,y\in \fg_\od \text{ and } z \in \wedge^{n-2}\fg \}.
\end{array}
\end{equation}
We define the {\it differential} ${\mathfrak d}^{-1}:XC^{-1}(\fg,M)\rightarrow XC^0(\fg,M)$ to be the trivial map. The {\it differential} ${\mathfrak d}^0$ is given by  \[
{\mathfrak d}^0:XC^0(\fg,M)\rightarrow XC^1(\fg,M) \quad m \mapsto {\mathfrak d}^0(m),
\]
where ${\mathfrak d}^0(m)(x)=x\cdot m. $
The {\it differential} ${\mathfrak d}^1$ is given by \[
{\mathfrak d}^1:XC^1(\fg,M)\rightarrow XC^2(\fg,M) \quad c \mapsto (dc, \mathfrak{q}), 
\] where 
\begin{equation}
\label{diffsuper}
\begin{array}{lcl}
dc(x,z) & = & c([x,z]) + x \cdot c(z) + z \cdot c(x)\quad \text{for all $x,z\in \fg$};\\[1mm]
\fq(x) & = & c(s(x))+x\cdot c(x) \quad \text{for all $x\in \fg_\od$}.
\end{array}
\end{equation}
Now, for $n\geq 2$ the {\it differential} ${\mathfrak d}^n$ is given by \[
{\mathfrak d}^n:XC^n(\fg,M)\rightarrow XC^{n+1}(\fg,M) \quad (c,\mathfrak{p}) \mapsto (d^n c, d^n\mathfrak{p}), 
\] where 
\begin{equation}
\begin{array}{lcl}
d^nc(z_1,\ldots, z_{n+1}) & = & \displaystyle \sum_{1 \leq i \leq n+1} z_i \cdot c(z_1, \ldots,\hat{z_i},\ldots,z_{n+1}) \\[4mm]
&&\displaystyle + \sum_{1 \leq i< j \leq n+1}c([z_i, z_j], z_1, \ldots, \hat{z_i}, \ldots, \hat{z_j},\ldots z_{n+1}),\\[2mm]
d^n{\mathfrak p}(x,z_1,\ldots,z_{n-1}) & = & \displaystyle x\cdot c(x,z_1,\ldots, z_{n-1})+\sum_{1 \leq i \leq n-1}z_i \cdot {\mathfrak p}(x,z_1, \ldots, \hat{z_i}, \ldots, z_{n-1})\\[3mm]
&&+c(s(x), z_1, \ldots, z_{n-1}) + \displaystyle\sum_{1 \leq i \leq n-1} c([x,z_i],x, z_1, \ldots, \hat{z_i}, \ldots, z_{n-1})\\[2mm]
&&\displaystyle + \sum_{1\leq i<j \leq n-1}{\mathfrak p}(x, [z_i, z_j],z_1,\ldots,\hat{z_i},\ldots, \hat{z_j},\ldots, z_{n-1}).
\end{array}
\end{equation}


\begin{Theorem}\label{welldefId}
The maps ${\mathfrak d}^n$ is well defined. Moreover, for all integers $n$
$$ {\mathfrak d}^{n+1}\circ {\mathfrak d}^n=0. $$ 
Hence, the pair $(XC^*(\fg,M),{\mathfrak d}^*)$ defines a cohomology complex for Lie superalgebras in characteristic 2.
\end{Theorem}

The proof of the theorem will be given next  when considering the cohomology of Hom-Lie superalgebras that reduce to ordinary Lie superalgebras when the structure map is the identity.

\subsection{Elucidation for $n=2,3$} 
Let us first exhibit the sets of cochains in the case where  $n=2,3$. 

If $v\in M$ and $a,b\in \fg^*_\od$, we can define the cochain $(v\otimes a\wedge b,\mathfrak q )\in XC^2(\fg,M)$ such that the quadratic form is ${\mathfrak q}(x)=a(x) b(x)\, v$ for all $x\in \fg_\od$. The polar form\footnote{Recall that to each quadratic from $\fq$ with values in a space $M$, its polar form is the bilinear form with values in $M$ given by
\[
B_\fq(x,y):=\fq(x+y)+\fq(x)+\fq(y).
\]} associated to $\fq$ is 
\[
B_\fq(x,y) = (a(x)b(y) + a(y)b(x)) \, v \text{ for all $x,y\in \fg_\od$}.
\] 
In particular, we can define the cochain $c=v\otimes a\wedge a$, where $q(x) = v (a(x))^2$ for all $x\in \fg_\od$  and $c(x,y) = 0$ for all $x,y\in \fg$.

Similarly, if $v\in M$, and $a,b\in \fg^*_\od$ but $c\in \fg$, we can define the cochain $(v\otimes a\wedge b \wedge c, {\mathfrak p} ) \in XC^2(\fg,M)$ such that the map $\mathfrak{p}$ is 
\[
{\mathfrak p}(x,z)=(a(x) b(x) c(z)+a(z) b(x) c(x)+ a(x)b(z)c(x))\, v \quad \text{for all $x\in \fg_\od$ and $z\in \fg$}
\]
Now, a direct computation shows that
\[
\begin{array}{lcl}
\mathfrak{p}(x+y,z)+\mathfrak{p}(x,z)+\mathfrak{p}(y,z) & = &  (a(x)b(y)c(z) + a(y)b(x) c(z)+ a(z)b(x) c(y)) \, v\\[2mm]
& & + (a(z) b(y) c(x)+b(z)a(x)c(y)+b(z)a(y)c(x))v\\[2mm]
&=& v \otimes (a\wedge b \wedge c)(x,y,z).
\end{array}
\] 

A 1-cocycle $c$ on $\fg$ with values in a $\fg$-module $M$ must satisfy the following conditions:
\begin{eqnarray}
\label{Cond1} x\cdot c(z)+z\cdot c(x)+c([x,z])&=& 0 \quad \text{for all $x,z\in \fg$},\\[2mm]
\label{Cond2} x\cdot c(x)+c(s(x)) &=& 0\quad \text{for all $x\in \fg_\od$}.
\end{eqnarray}
A 2-cocycle $(c,\fq)$ on $\fg$ with values in $M$ must satisfy the following conditions:
\begin{eqnarray}
\label{2-cocCond1} 
0 &=& x\cdot c(y,z)+c([x,y],z)+ \circlearrowleft (x,y,z) \quad \text{for all $x,y,z\in \fg$},\\[2mm]
0 &=&  x\cdot c(x,z)+z \cdot \fq(x)+ c(s(x),z)+ c([x,z],x]\\[2mm]
\nonumber &&\text{for all $x\in \fg_\od$ and for all $z\in \fg$},
\end{eqnarray}
\subsection{Cohomology  of Hom-Lie superalgebras in characteristic 2} Let  $( \fg, [\cdot,\cdot],s, \alpha)$ be a Hom-Lie superalgebra in characteristic 2 and $(M,\beta)$ be a $\fg$-module, see Definition \ref{defmodalpha}. The space of $n$-cochains are defined similarly to  (\ref{xcochains}) with a slight  difference with respect to degree 0 space and  an extra condition that is
\begin{equation}
    \label{cohalphabeta}
\beta \circ c =c \circ (\alpha\wedge \cdots \wedge \alpha), \quad \text{ and } \quad  \beta \circ {\mathfrak p} ={\mathfrak p} \circ (\alpha\wedge \cdots \wedge \alpha).
\end{equation}
\begin{equation}\label{xcochains2}
\begin{array}{lcl}
XC_\alpha ^{-1}(\fg;M)&:=& \{0\},\\[2mm]
XC_\alpha ^{0}(\fg;M)&:=& \{m \in M \; | \; \beta(m)=m \text{ and } \alpha(x)\cdot (y\cdot m)=x\cdot (y \cdot m)\text{ for all $x,y\in \fg$} \},\\[2mm]
XC_\alpha^1(\fg;M)&:=&\{c \;| \text{ where $c:\fg \rightarrow M$ is linear and satisfies Eq. \ref{cohalphabeta}} \}, \\[2mm]
 XC_\alpha^n(\fg;M)&:=& \{ (c,\mathfrak{p}) \; | \text{ where } c:\wedge^n \fg \rightarrow M  \text{ is a multi-linear map satisfying Eq.\ref{cohalphabeta} and }  \\[1mm]
&&   \mathfrak p:\fg_\od \times  \wedge^{n-2} \fg \rightarrow M  \text{ is a map as in \eqref{mapp} satisfying Eq. \ref{cohalphabeta} } \text{ such that }\\[1mm]
&&
\mathfrak{p}(x+y,z)+\mathfrak{p}(x,z)+\mathfrak{p}(y,z)=c(x,y,z)   \text{ for all }  x,y\in \fg_\od \text{ and } z \in \wedge^{n-2}\fg \}.
\end{array}
\end{equation}

1-cochains are just linear functions $c$ on $\fg$ with values in a $\fg$-module $M$ such that $\beta\circ c=c \circ \alpha$. Let us define the {\it differentials} in our context. First, let us define define $\mathfrak{d}_\alpha^0$ and $\mathfrak{d}_\alpha^1$.
\[
{\mathfrak d}^0_\alpha:XC^0_\alpha(\fg,M)\rightarrow XC^1_\alpha(\fg,M) \quad m \mapsto d^0_\alpha m, 
\] 
where $d^0_\alpha m(x)=x\cdot m$ for all $x\in \fg$.  Additionally, 
\[
{\mathfrak d}^1_\alpha:XC^1_\alpha(\fg,M)
\rightarrow XC^2_\alpha (\fg,M) \quad c \mapsto (d^1_\alpha c, \fq ), 
\] 
where 
\begin{equation}
\label{diffsuperalpha}
\begin{array}{lcl}
d^1_\alpha c(x,z) & = & c([x,z]) + x \cdot c(z) + y \cdot c(x)\quad \text{for all $x,z\in \fg$};\\[1mm]
\fq (x) & = & c(s(x))+ x\cdot c(x) \quad \text{for all $x\in \fg_\od$}.
\end{array}
\end{equation}

Note that these definitions are consistent as  showed by the following Lemma.
\begin{Proposition} 
The differentials ${\mathfrak d}^0_\alpha$ and ${\mathfrak d}^1_\alpha $ are indeed well-defined; namely, $\text{\rm{Im}}\,({\mathfrak d}^0_\alpha)\subseteq XC^1_\alpha(\fg, M)$ and  $\text{\rm{Im}}\,({\mathfrak d}^1_\alpha)\subseteq XC^2_\alpha(\fg, M)$.
\end{Proposition}
\begin{proof} Let us first deal with ${\mathfrak d}^0_\alpha$. We have
\[
d^0_\alpha m(\alpha(x))=\alpha(x)\cdot m= \alpha(x)\cdot \beta(m)=\beta(x\cdot m)= \beta(d^0_\alpha m(x)).
\]
Therefore, Eq. \ref{cohalphabeta} is satisfied. Let us now deal deal with ${\mathfrak d}^1_\alpha$. We will only prove that $\mathfrak{q}$ satisfies Eq. \ref{cohalphabeta}. Indeed, 
\[
\begin{array}{lcl}
\mathfrak{q} (\alpha(x)) & = & c(s(\alpha(x)))+\alpha(x)\cdot c(\alpha(x))=c(\alpha(s(x)))+\alpha(x)\cdot \beta(c(x))\\[2mm]
& = & \beta(c(s(x)))+\beta(x\cdot c(x))=  \beta({\mathfrak q}(x)).
\end{array}
\]
On the other hand, we have
\[
\begin{array}{lcl}
\fq (x+y) + \fq (x) + \fq (y) & = & c( s(x+y)) + (x+y)\cdot c(x+y) +c(s(x))  \\[2mm]
&& + x \cdot c(x) + c(s(y))+ y \cdot c(y) \\[2mm]
& = & c([x,y]) + y \cdot c(x) + x \cdot c(y)=d^1_\alpha c(x,y).\qedhere
\end{array}
\]\end{proof}
A 1-cocycle $c$ on $\fg$ with values in a $\fg$-module $M$ must satisfy the following conditions:
\begin{eqnarray}
\label{Cond1alpha} x\cdot c(y)+y\cdot c(x)+c([x,y])&=& 0 \quad \text{for all $x,y\in \fg$},\\[2mm]
\label{Cond2alpha} x\cdot c(x)+c(s(x)) &=& 0\quad \text{for all $x\in \fg_\od$}.
\end{eqnarray}
The space of all 1-cocycles is denoted by $Z^1_\alpha(\fg;M)$.

Now, for $n \geq 2$ the {\it differential} ${\mathfrak d}^n_\alpha $ is given by \[
{\mathfrak d}^n:XC^n_\alpha (\fg,M)\rightarrow XC^{n+1}_\alpha(\fg,M) \quad (c,\mathfrak{p}) \mapsto (d^n_\alpha c, d^n_\alpha \mathfrak{p}), 
\] where 
\begin{align*}
& d^n_\alpha c(z_1,\ldots, z_{n+1})  =  \displaystyle \sum_{1 \leq i \leq n+1} \alpha^{n-1}(z_i) \cdot c(z_1, \ldots,\hat{z_i},\ldots,z_{n+1}) \\[4mm]
 &\displaystyle + \sum_{1 \leq i< j \leq n+1}c([z_i, z_j], \alpha(z_1), \ldots, \hat{z_i}, \ldots, \hat{z_j},\ldots \alpha(z_{n+1})),\\[2mm]
&d^n_\alpha {\mathfrak p}(x,z_1,\ldots,z_{n-1})  =   \displaystyle \alpha^{n-1}(x)\cdot c(x,z_1,\ldots, z_{n-1}) +c(s(x), \alpha(z_1), \ldots, \alpha(z_{n-1}))\\[3mm]
&\displaystyle +\sum_{1 \leq i \leq n-1}\alpha^{n-1}(z_i) \cdot {\mathfrak p}(x,z_1, \ldots, \hat{z_i}, \ldots, z_{n-1})\\[2mm]
&+ \displaystyle\sum_{1 \leq i \leq n-1} c([x,z_i],\alpha(x), \alpha(z_1), \ldots, \hat{z_i}, \ldots, \alpha(z_{n-1}))\\[3mm]
&\displaystyle + \sum_{1\leq i<j \leq n-1}{\mathfrak p}(\alpha(x), [z_i, z_j],\alpha(z_1),\ldots,\hat{z_i},\ldots, \hat{z_j},\ldots, \alpha(z_{n-1})).
\end{align*}
In particular for $n=2$, the {\it differential} is given by  \[
{\mathfrak d}^2_\alpha:XC^2_\alpha(\fg,M)\rightarrow XC^3_\alpha(\fg,M) \quad (c,{\mathfrak p}) \mapsto (d^2_\alpha c, d^2_\alpha  \mathfrak{p}), 
\] where 
\begin{equation*}
\label{diffsuper2alpha}
\begin{array}{lcl}
d^2_\alpha c(z_1,z_2,z_3) & = & \alpha(z_1)\cdot c(z_2,z_3)+c([z_1,z_2],\alpha(z_3))+ \circlearrowleft (z_1,z_2,z_3) \quad \text{for all $z_1,z_2,z_3\in \fg$};\\[1mm]
d^2_\alpha {\mathfrak p}(x,z_1) & = & \alpha(x)\cdot c(x,z_1)+ \alpha(z_1) \cdot {\mathfrak p}(x)+ c(s(x),\alpha(z_1))+ c([x,z_1],\alpha(x)) \\[2mm]
&& \text{for all $x\in \fg_\od,$ and for all $z_1\in \fg$}.
\end{array}
\end{equation*} 
A 2-cocycle is 2-tuple $(c,\mathfrak{p})$ satisfying the following conditions:
\begin{eqnarray}
\label{2-cocCond1alpha} 
0 &=& \alpha(z_3) \cdot c(z_1,z_2)+c([z_1,z_2],\alpha(z_3))+ \circlearrowleft (z_1,z_2,z_3) \quad \text{for all $z_1,z_2,z_3\in \fg$},\\[2mm]
\label{2-cocCond2alpha}  0 &=&  \alpha(x)\cdot c(x,z_1)+ \alpha(z_1) \cdot {\mathfrak p}(x)+ c(s(x),\alpha(z_1))+ c([x,z_1],\alpha(x)]\\[2mm]
\nonumber &&\text{for all $x\in \fg_\od$ and for all $z_1\in \fg$},
\end{eqnarray}
The first step here is to show that the map ${\mathfrak d}_\alpha^n$ is well defined, for every twist $\alpha$. By doing so, we give a proof to the first part of Theorem \ref{welldefId} in the case where $\alpha=\mathrm{id}$. 
\begin{Proposition} The map $\mathfrak{d}^n_\alpha$ is well-defined; namely, $\text{\rm{Im}}\,({\mathfrak d}^n_\alpha)\subseteq XC^{n+1}_\alpha (\fg, M)$.
\end{Proposition}
\begin{proof} For all $x,y \in \fg_\od$ and for all $z_1,\ldots,z_n\in \fg$, 
we have 
\begin{align*}
&d^n_\alpha {\mathfrak p}(x+y,z_1,\ldots,z_n)\\[2mm] 
&= 
\displaystyle \alpha^{n-1}(x+y)\cdot c(x+y,z_1,\ldots, z_{n-1})+\sum_{1 \leq i \leq n-1}\alpha^{n-1}(z_i) \cdot {\mathfrak p}(x+y,z_1, \ldots, \hat{z_i}, \ldots, z_{n-1})\\[3mm]
&+c(s(x+y), \alpha(z_1), \ldots, \alpha(z_{n-1})) + \displaystyle\sum_{1 \leq i \leq n-1} c([x+y,z_i],\alpha(x+y), \alpha(z_1), \ldots, \hat{z_i}, \ldots, \alpha(z_{n-1}))\\[2mm]
&\displaystyle + \sum_{1\leq i<j \leq n-1}{\mathfrak p}(\alpha(x+y), [z_i, z_j],\alpha(z_1),\ldots,\hat{z_i},\ldots, \hat{z_j},\ldots, \alpha(z_{n-1})).
\\
& = d^n_\alpha {\mathfrak p}(x,z_1,\ldots,z_n)+ d^n_\alpha {\mathfrak p}(y,z_1,\ldots,z_n)+
\alpha^{n-1}(x)\cdot c(y,z_1,\ldots, z_{n-1})\\[3mm]
&+\alpha^{n-1}(y)\cdot c(x,z_1,\ldots, z_{n-1})+\sum_{1 \leq i \leq n-1}\alpha^{n-1}(z_i) \cdot c(x,y,z_1, \ldots, \hat{z_i}, \ldots, z_{n-1})\\[3mm]
&+c([x,y], \alpha(z_1), \ldots, \alpha(z_{n-1})) + \displaystyle\sum_{1 \leq i \leq n-1} c([x,z_i],\alpha(y), \alpha(z_1), \ldots, \hat{z_i}, \ldots, \alpha(z_{n-1}))\\[3mm]
& + \displaystyle\sum_{1 \leq i \leq n-1} c([y,z_i],\alpha(x), \alpha(z_1), \ldots, \hat{z_i}, \ldots, \alpha(z_{n-1}))\\
&\displaystyle + \sum_{1\leq i<j \leq n-1}c (\alpha(x), \alpha(y), [z_i, z_j],\alpha(z_1),\ldots,\hat{z_i},\ldots, \hat{z_j},\ldots, \alpha(z_{n-1}))
\\
& = d^n_\alpha {\mathfrak p}(x,z_1,\ldots,z_n)+ d^n_\alpha {\mathfrak p}(y,z_1,\ldots,z_n)+d^n_\alpha c(x,y,z_1,\ldots, z_n),
\end{align*}
where we have used the fact  that $s(x+y)  =  s(x)+s(y)+[x,y]$ and  \[
\begin{array}{lcl}
{\mathfrak p}(x+y,z_1,\ldots,z_{n-1})+{\mathfrak p}(x,z_1, \ldots,z_{n-1})+{\mathfrak p}(y,z_1,\ldots,z_{n-1})&=&c(x,y,z_1,\ldots, z_{n-1}).\qedhere
\end{array}
\]
\end{proof}

\begin{Theorem}\label{d2=0}
For all $n\geq 1$, we have $\mathfrak{d}_\alpha^{n}\circ \mathfrak{d}_\alpha^{n-1}=0$.
Hence, the pair $(XC^*_\alpha(\fg,M),{\mathfrak d}_\alpha^*)$ defines a cohomology complex for Hom-Lie superalgebras in characteristic 2.
\end{Theorem}
In order to prove this theorem, we will need the following Lemma.
\begin{Lemma} \label{lem74} If $(c, {\mathfrak p}) \in XC^n_\alpha(\fg,M)$, then

\textup{(}i\textup{)} $\alpha^{n-1}(x) \cdot ( \alpha^{n-2}(x) \cdot c(z_1,\ldots,z_n))=\alpha^{n-2}(s(x))\cdot c(\alpha(z_1),\ldots, \alpha(z_n))$ for all $x\in \fg_\od$ and for all $z_1,\ldots,z_n\in \fg$.

\textup{(}ii\textup{)} $\alpha(z_i)\cdot (z_j\cdot c(z_1,\ldots, z_n))+\alpha(z_j)\cdot (z_i\cdot c(z_1,\ldots, z_n))=[z_i,z_j]\cdot c(\alpha(z_1), \ldots \alpha(z_n)) $ for all $z_1,\ldots z_n\in \fg$.
\end{Lemma}
\begin{proof}
Let us only prove Part (i). Using the fact that $\beta \circ c =c \circ (\alpha\wedge \cdots \wedge \alpha)$ we get
\begin{align*}
&\alpha^{n-1}(x) \cdot ( \alpha^{n-2}(x) \cdot c(z_1,\ldots,z_n))=  s(\alpha^{n-2}(x)) \cdot \beta (c(z_1,\ldots, z_n))\\[2mm]
&= \alpha^{n-2}(s(x)) \cdot  (c(\alpha(z_1),\ldots, \alpha(z_n))).\qedhere
\end{align*} \end{proof}
\begin{proof}[Proof of Theorem \ref{d2=0}] Lert us first show that $\mathfrak{d}^1_\alpha \circ \mathfrak{d}^0_\alpha=0$. Indeed, for all $x,y \in \fg$ and $m\in XC_\alpha ^{0}(\fg;M)$ we have 
\[
\begin{array}{lcl}
d^1_\alpha \circ d^0_\alpha m (x,y)&=&x \cdot d^0_\alpha m(y)+ y \cdot d^0_\alpha m(x)+ d^0_\alpha m ([x,y])=x \cdot (y \cdot m)+ y \cdot (x \cdot m)+[x,y]\cdot m\\[1mm]
&=&\alpha(x) \cdot (y \cdot m)+\alpha(y) \cdot (x \cdot m)+ [x,y]\cdot \beta(m)=0.
\end{array}
\]
On the other hand, for all $x\in \fg_\od$ and $m\in XC_\alpha ^{0}(\fg;M)$, we have 
\[
{\mathfrak q}(x)=d^0_\alpha m(s(x))+ x \cdot d^0_\alpha m(x)= s(x)\cdot m+ x\cdot (x \cdot m)= s(x)\cdot \beta(m)+\alpha(x)\cdot (x \cdot m)=0.
\]
Let us now show that $\mathfrak{d}^n_\alpha \circ \mathfrak{d}^{n-1}_\alpha=0$ for all $n>1$.  To show that $d^n_\alpha \circ d^{n-1}_\alpha (c)=0$ is a routine, see for instance \cite{AMS}. Let us show that $d^n_\alpha \circ d^{n-1}_\alpha ({\mathfrak p})=0$. This would imply that ${\mathfrak d}^n_\alpha \circ {\mathfrak d}^{n-1}_\alpha(c, {\mathfrak p})=0$. Actually, the computation is very cumbersome so we will break it into small pieces. First, we compute:
\begin{align*}
& d^n_\alpha \circ d^{n-1}_\alpha ({\mathfrak p})(x,z_1,\ldots,z_{n-1}) =\alpha^{n-1}(x)\cdot d^{n-1}_\alpha  c(x, z_1, \ldots, z_{n-1})\\[2mm]
&+\sum_{i=1}^{n-1}\alpha^{n-1}(z_i) d^{n-1}_\alpha {\mathfrak p} (x,z_1,\ldots, \widehat{z_i},\ldots, z_{n-1})+ d^{n-1}_\alpha c(s(x), \alpha(z_1), \ldots, \alpha(z_{n-1}))\\[2mm]
&+ 
\sum_{i=1}^{n-1} d^{n-1}_\alpha c ([x,z_i], \alpha(x),\alpha(z_1),\ldots, \widehat{\alpha(z_i)}, \ldots, \alpha(z_{n-1}))\\[2mm]
& + \sum_{1\leq i<j\leq n-1} d^{n-1}_\alpha  {\mathfrak p}(\alpha(x), [z_i, z_j], \alpha(z_1), \ldots \widehat{\alpha(z_i)}, \ldots, \widehat{\alpha(z_j)}, \ldots, \alpha(z_{n-1})). 
\end{align*}
There are five terms in the expression above. We will compute each term  separately.
\begin{small}
\begin{align*}
&\text{\underline{Part 1: }}\sum_{1\leq i \leq n-1} \alpha^{n-1}_\alpha (z_i) d^{n-1} {\mathfrak p}(x,z_1,\ldots, \widehat{z_i},\ldots,z_{n-1})=\\[2mm]
&+\sum_{1\leq i \leq n-1} \alpha^{n-1}(z_i) \cdot [ \alpha^{n-2}(x)\cdot c(x,z_1,\ldots,\widehat{z_i}, \ldots, z_{n-1})\\[2mm] &+\sum_{1\leq j\leq n-2} \alpha^{n-2}(z_j) {\mathfrak p} (x,z_1,\ldots, \widehat{z_i},\ldots, \widehat{z_j},\ldots, z_{n-1})+ c(s(x), \alpha(z_1),\ldots, \widehat{\alpha(z_i)},\ldots, \alpha(z_{n-1}))\\[2mm]
&+ \sum_{1\leq j \leq n-2}c([x,z_i], \alpha(x), \alpha(z_1), \ldots, \widehat{\alpha(z_i)},\ldots, \widehat{\alpha(z_j)},\ldots,\alpha(z_{n-1}))\\[2mm]
& + \sum_{1\leq l < j \leq n-2} {\mathfrak p} (\alpha(x), [z_l,z_j], \alpha(z_1), \ldots,\widehat{\alpha(z_l)}, \ldots, \widehat{\alpha(z_j)}, \ldots, \widehat{\alpha(z_i)},\ldots, \alpha(z_{n-1}))].
\end{align*}
\begin{align*}
  & \text{\underline{Part 2: }}  \alpha^{n-1}(x)d^{n-1}_\alpha  c(x, z_1, \ldots, z_{n-1}) = \alpha^{n-1}(x) \cdot [ \alpha^{n-2}(x) c (z_1, \ldots, z_{n-1}) \\[2mm]
  &+ \sum_{1\leq i \leq n-1} \alpha^{n-2}(z_i) c(x,z_1, \ldots, \widehat{z_i}, \ldots, z_{n-1}) +
  \sum_{1\leq i \leq n-1} c([x,z_i], \alpha(z_1), \ldots, \widehat{\alpha(z_i)}, \ldots, \alpha(z_{n-1}))\\[2mm]
 &+ \sum_{1\leq i < j \leq n-1} c([z_i, z_j], \alpha(x), \alpha(z_1),\ldots, \widehat{\alpha(z_i)},\ldots,  \widehat{\alpha(z_j)},\ldots, \alpha(z_{n-1}))].
\end{align*}
\begin{align*}
& \text{\underline{Part 3: }}d^{n-1}_\alpha c(s(x), \alpha(z_1),\ldots, \alpha(z_{n-1}))=  \alpha^{n-2}(s(x))\cdot c(\alpha(z_1),\ldots,\alpha(z_{n-1}))
\\[2mm]
&+\sum_{1\leq i \leq n-1} \alpha^{n-1}(z_i)\cdot c(s(x), \alpha(z_1), \ldots, \widehat{\alpha(z_i)},\ldots, \alpha(z_{n-1}))\\[2mm]
&\sum_{1\leq i < j \leq n-1} c([\alpha(z_i), \alpha(z_j)], \alpha(s(x)), \alpha^2(z_1), \ldots, \widehat{\alpha^2(z_i)}, \ldots, \widehat{\alpha^2(z_j)},\ldots, \alpha^2(z_{n-1})) \\[2mm]
&+ \sum_{1\leq i \leq n-1} c([s(x), \alpha(z_i)], \alpha^2(z_1), \ldots, \widehat{\alpha^2(z_i)}, \ldots, \alpha^2(z_{n-1})).
\end{align*}
\begin{align*}
&\text{\underline{Part 4: }}\sum_{1\leq i \leq n-1} d^{n-1}_\alpha c ([x,z_i], \alpha(x), \alpha(z_1),\ldots, \widehat{\alpha(z_i)}, \ldots, \alpha(z_{n-1}))\\[2mm]
&=\sum_{1\leq i \leq n-1} [\alpha^{n-2}([x,z_i])\cdot c(\alpha(x), \alpha(z_1), \ldots, \widehat{\alpha(z_i)},\ldots, \alpha(z_{n-1}))\\[2mm]
&+\alpha^{n-1}(x)\cdot c([x,z_i], \alpha(z_1),\ldots, \widehat{\alpha(z_i)}, \ldots, \alpha(z_{n-1}))\\[2mm]
&+ \sum_{1\leq l \leq n-2} \alpha^{n-1}(z_l)\cdot c([x,z_i], \alpha(x), \alpha(z_1), \ldots, \widehat{\alpha(z_l)},\ldots,\widehat{\alpha(z_i)},\dots, \alpha(z_{n-1}))\\[2mm]
&+c([[x,z_i],\alpha(x)], \alpha^2(z_1), \ldots, \widehat{\alpha^2(z_i)}, \ldots, \alpha^2(z_{n-1}))\\[2mm]
&+ \sum_{1\leq j \leq n-1} c([[x,z_i], \alpha(z_j)], \alpha^2(x), \alpha^2(z_1), \ldots, \widehat{\alpha^2(z_i)},\ldots,  \widehat{\alpha^2(z_j)},\ldots  \alpha^2(z_{n-1}))\\[2mm]
&+ \sum_{1\leq l \leq n-1} c([\alpha(x), \alpha(z_l)], \alpha([x,z_i]), \alpha^2(z_1), \ldots, \widehat{\alpha^2(z_i)},\ldots,  \widehat{\alpha^2(z_l)},\ldots  \alpha^2(z_{n-1}))\\[2mm]
&+ \sum_{1\leq u <v \leq n-1} c(\alpha([x,z_i]),\alpha^2(x), [\alpha(z_u), \alpha(z_v)], \alpha^2(z_1), \ldots, \widehat{\alpha^2(z_i)},\ldots, \widehat{\alpha^2(z_u)},\ldots  \widehat{\alpha^2(z_v)},\ldots  \alpha^2(z_{n-1}))].
\\
  & \text{\underline{Part 5 :} }\sum_{1\leq i <j \leq n-1} d^{n-1}_\alpha  {\mathfrak p}(\alpha(x), [z_i, z_j], \alpha(z_1), \ldots, \widehat{\alpha(z_i)}, \ldots, \widehat{\alpha(z_j)}, \ldots, \alpha(z_{n-1}))\\[2mm]
  &= \sum_{1\leq i <j \leq n-1}  [ \alpha^{n-1}(x)\cdot c(\alpha(x), [z_i, z_j] , \alpha(z_1), \ldots, \widehat{\alpha(z_i)},\ldots, \widehat{\alpha(z_j)},\ldots, \alpha(z_{n-1}))\\[2mm]
  &+\sum_{1\leq l \leq n-1} \alpha^{n-1}(z_l) {\mathfrak p}(\alpha(x), [z_i, z_j], \alpha(z_1), \ldots, \widehat{\alpha(z_i)}, \ldots, \widehat{\alpha(z_j)}, \ldots,\widehat{\alpha(z_l)}, \ldots, \alpha(z_{n-1}))\\[2mm]
  &+ \alpha^{n-2}([z_i, z_j])\cdot {\mathfrak p}(x, \alpha(z_1), \ldots, \widehat{\alpha(z_i)}, \ldots, \widehat{\alpha(z_j)}, \ldots, \alpha(z_{n-1}))\\[2mm]
  &+ c(s(\alpha(x)), \alpha([z_i, z_j]), \alpha^2(z_1),\ldots, \widehat{\alpha(z_i)}, \ldots, \widehat{\alpha(z_j)}, \ldots, \alpha(z_{n-1}))\\[2mm]
  &\sum_{1\leq l \leq n-1} c([\alpha(x), \alpha(z_l)], \alpha([z_i, z_j]), \alpha^2(x), \alpha^2(z_1),\ldots, \widehat{\alpha^2(z_i)}, \ldots, \widehat{\alpha^2(z_j)}, \ldots, \alpha^2(z_{n-1}))\\[2mm]
  &+ c([\alpha(x), [z_i, z_j]], \alpha^2(z_1), \ldots\widehat{\alpha^2(z_i)}, \ldots, \widehat{\alpha^2(z_j)}, \ldots, \widehat{\alpha^2(z_l)},\ldots, \alpha^2(z_{n-1}))\\[2mm]
  &+\sum_{1\leq l \leq n-1}{\mathfrak p}(\alpha^2(x), [\alpha(z_l), [z_i, z_j]], \alpha^2(z_1), \ldots\widehat{\alpha^2(z_i)}, \ldots, \widehat{\alpha^2(z_j)}, \ldots,\widehat{\alpha^2(z_l)}, \ldots \alpha^2(z_{n-1}))\\[2mm]
  &+\sum_{1\leq u<v \leq n-1} T_{ijuv}],
\end{align*}
\end{small}
where 
\begin{small}
\[
T_{ijuv}={\mathfrak p}(\alpha^2(x), \alpha([z_i, z_j]),[\alpha(z_u), \alpha(z_v)], \alpha^2(z_1), \ldots\widehat{\alpha^2(z_i)}, \ldots, \widehat{\alpha^2(z_j)}, \ldots,\widehat{\alpha^2(z_u)}, \ldots,\widehat{\alpha^2(z_v)},
\ldots).
\]
\end{small}
Now, using Lemma \ref{lem74} a direct computation shows that
\[
\text{Part 1+Part 2+Part 3+Part 4+Part 5}=0.\qedhere
\] \end{proof}
Now, we are ready to define a cohomology of Hom-Lie superalgebras in characteristic 2. The kernel of the map  ${\mathfrak d}^n_\alpha$, denoted by $Z^n_\alpha(\fg;M)$,  is the space of $n$-cocycles. The range of the map ${\mathfrak d}^{n-1}_\alpha$, denoted by $B^n_\alpha(\fg;M)$, is the space of coboundaries. 

We define the $n^{th}$ cohomology space as 
\[
{\mathrm H}^n_\alpha(\fg;M):=Z^n_\alpha(\fg;M)/B^n_\alpha(\fg;M).
\]
\begin{Remark}{\rm 
    The cohomology defined above coincides when $\alpha=\mathrm{id}_\fg$ and $\beta=\mathrm{id}_M$, with  the cohomology of Lie superalgebras in characteristic 2 defined in the previous section.}
\end{Remark}

\begin{Example}{\rm  We compute the second cohomology of the Hom-Lie superalgebra $\mathfrak{oo}_{I\Pi}^{(1)}(1|2)_\alpha$ defined in Example \ref{examoo}.  We will assume here that the field $\mathbb K$ is infinite.

(i) The cohomology space $\mathrm{H}^2_\alpha(\mathfrak{oo}_{I\Pi}^{(1)}(1|2)_\alpha;\mathbb{K})$ is trivial. Recall that in this case the map $\beta=\mathrm{Id}$.

Let us first show that cocycles of  the form $(0,\mathfrak{p})$ are necessarily trivial. In fact, the condition $\mathfrak{p}=\mathfrak{p}\circ \alpha$ and $\varepsilon \neq 0$ implie that 
\[
{\mathfrak p}(x_1)=0 \quad \text{and} \quad  {\mathfrak p}(y_1)=m \; \text{ (arbitrary)}.
\]
Choose $B_m=m \, y_2^*$, where $m\in \mathbb{K}$. A direct computation shows that $d^2_\alpha B_m=0$. Let us compute the corresponding ${\mathfrak q}_m$. Indeed, 
\[
{\mathfrak q}_m(x_1)=B_m(s(x_1))=B_m(x_2)=0,
\]
and 
\[
{\mathfrak q}_m(y_1)=B_m(s(y_1))=B_m(
\varepsilon h_1+ \varepsilon^2 x_2+ y_2)=m .
\]
It follows that $(0,\mathfrak{p})= (d^2_\alpha B_m, \mathfrak{q}_m)$ and hence its cohomlogy class is trivial.

Let us now describe 2-cocycles of the form $(c,\mathfrak{p})$. A direct computation shows that 
\[
c_1=x_1^*\wedge y_1^*+ x_2^*\wedge y_2^*, \quad c_2=h_1^*\wedge y_1^*+x_1^*\wedge y_2^*,
\]
are the only cochains verifying both conditions $c_i\circ (\alpha \wedge \alpha)=c_i$ and $d_\alpha^2c_i=0$ for $i=1,2$. Let us describe the corresponding $\mathfrak{p}$'s. We have 
\[
\begin{array}{lcllcl}
\mathfrak{p}_1(x_1)&=& \varepsilon^{-1}, & \mathfrak{p}_1(y_1) &= & m_1 \quad  \text{(arbitrary)} \\[2mm]
\mathfrak{p}_2(x_1) &= & 0 & \mathfrak{p}_2(y_1) & = & m_2 \quad  \text{(arbitrary)}
\end{array}
\]
We then get that $\mathfrak{d}^2_\alpha(c_1, \mathfrak{p}_1)=
\mathfrak{d}^2_\alpha(c_2,\mathfrak{p}_2)=0$.

Let us now describe the coboundaries. Choose $b_1=h_1^*+\varepsilon^{-1}x_2^*$. It follows that 
\[
d^1_\alpha b_1= x_1^*\wedge y_1^*+ x_2^*\wedge y_2^*.
\]
Now,
\[
\mathfrak{q}_1(x_1)=b_1(s(x_1))=b_1(x_2)=
\varepsilon^{-1},
\]
and
\[
\mathfrak{q}_1(y_1)=b_1(s(y_1))=b_1
(\varepsilon h+ \varepsilon^2 x_2 + y_2)=0.
\]
Choose $b_2=y_1^*$. A direct computation shows that 
\[
d^1_\alpha b_2=h_1^*\wedge y_1^*+ x_1^*\wedge y_2^*, \quad \text{and $\mathfrak{q}_2\equiv 0$. }
\]
It follows that 
\[
(c_1, \mathfrak{p}_1)=(d^2_\alpha b_1, \mathfrak{q}_1)+ (d^2_\alpha B_ {m_1}=0,\mathfrak{q}_{m_1}), \quad \text{ and } \quad (c_2,\mathfrak{p}_2)=(d^2_\alpha b_2,\mathfrak{q}_2)+
(d^2_\alpha B_{m_2}=0,\mathfrak{q}_{m_2}).
\]
Therefore, the cohomology space $\mathrm{H}^2_\alpha(\mathfrak{oo}_{I\Pi}^{(1)}(1|2)_\alpha; \mathbb{K})$ is trivial.

(ii) Let us now compute the cohomology space: $\mathrm{H}^2_\alpha(\mathfrak{oo}_{I\Pi}^{(1)}(1|2)_\alpha; \mathfrak{oo}_{I\Pi}^{(1)}(1|2)_\alpha)$. Recall that in the case where  $\alpha=\mathrm{Id}$, this cohomology space has only two non-trivial 2-cocycles.

\underline{The case where $\varepsilon \not =1$}: the space $\mathrm{H}^2_\alpha(\mathfrak{oo}_{I\Pi}^{(1)}(1|2)_\alpha; \mathfrak{oo}_{I\Pi}^{(1)}(1|2)_\alpha$ is generated by the  non-trivial 2-cocycles:
\[
(c_4, \mathfrak{p}_4), \quad (c_5, \mathfrak{p}_5), \quad (c_9, \mathfrak{p}_9), \quad (c_{10}, \mathfrak{p}_{10}), \quad (c_{11}, \mathfrak{p}_{11}),\quad (c_{12}, \mathfrak{p}_{12}),
\]
where 
\[
\begin{array}{lcl}
c_4 & = & \varepsilon x_1 \otimes h_1^*\wedge x_1^*+\varepsilon^2 x_1\otimes h_1^*\wedge y_1^*+ x_1 \otimes x_1^*\wedge x_2^*+ \varepsilon^2 x_1 \otimes x_1^*\wedge y_2^*+\varepsilon^2 x_2 \otimes x_1^*\wedge y_1^* \\[2mm]
&& \varepsilon^2x_2\otimes x_2^*\wedge y_2^*+ \varepsilon y_1 \otimes h_1^*\wedge y_1^*+ y_1\otimes x_2^*\wedge y_1^*,\\[2mm]
c_5 & = & h_1\otimes h_1^*\wedge y_2^*+ x_1\otimes h_1^* \wedge y_1^*+ x_1\otimes x_1^*\wedge y_2^*,\\[2mm] 
c_9 & = & h_1 \otimes h_1^*\wedge x_1^*+ \varepsilon h_1 \otimes h_1^*\wedge y_1^*+ \varepsilon x_1 \otimes x_1^*\wedge y_1^*+ \varepsilon x_2 \otimes h_1^*\wedge x_1^*+ \varepsilon^2 x_2 \otimes h_1^*\wedge y_1^*\\[2mm]
&&+ \varepsilon^2 x_2\otimes x_1^*\wedge y_2^*+ \varepsilon x_2 \otimes x_2^*\wedge y_1^*+y_2\otimes h_1^*\wedge y_1^*+ y_2\otimes x_1^*\wedge y_2^*,\\[2mm]
c_{10} & = & h_1\otimes x_1^*\wedge y_1^*,\\[2mm]
c_{11} & = & h_1\otimes x_2^*\wedge y_2^*+x_1\otimes x_1^*\wedge y_2^*+ x_1\otimes x_2^*\wedge y_1^*+ \varepsilon x_2\otimes x_1^*\wedge y_1^*+ \varepsilon x_2\otimes x_2^*\wedge y_2^*,\\[2mm]
c_{12} & = & x_1\otimes y_1^*\wedge x_1^*+ \varepsilon x_1 \otimes h_1^*\wedge y_1^*+ \varepsilon x_2 \otimes x_1^*\wedge y_1^*+ \varepsilon x_2\otimes x_2^*\wedge y_2^*\\[2mm]
&&+ y_1\otimes h_1^*\wedge y_1^*+ y_1 \otimes x_1^*\wedge y_2^*,
\end{array}
\]
and
\[
\begin{array}{lcllcl}
\mathfrak{p}_4(x_1) & = & h_1, & \mathfrak{p}_4(y_1) & = & \varepsilon^2 h_1 + \varepsilon^3 x_2+ \varepsilon y_2,\\[2mm]
\mathfrak{p}_5(x_1) & = & 0, & \mathfrak{p}_5(y_1) & = &0,\\[2mm]
\mathfrak{p}_9(x_1) & = & x_1, & \mathfrak{p}_9(y_1) & = &0,\\[2mm]
\mathfrak{p}_{10}(x_1) & = & x_2, & \mathfrak{p}_{10}(y_1) & = & \varepsilon h_1+  \varepsilon^2 x_2+y_2 ,\\[2mm]
\mathfrak{p}_{11}(x_1) & = & x_2, & \mathfrak{p}_{11}(y_1) & = & 0,\\[2mm]
\mathfrak{p}_{12}(x_1) & = & x_2, & \mathfrak{p}_{12}(y_1) & = & 0.
\end{array}
\]
\underline{The case where $\varepsilon=1$}: the space $\mathrm{H}^2_\alpha(\mathfrak{oo}_{I\Pi}^{(1)}(1|2)_\alpha; \mathfrak{oo}_{I\Pi}^{(1)}(1|2)_\alpha)$ is generated by the non-trivial 2-cocycles: 
\[
(c_4, \mathfrak{p}_4), \quad (c_5, \mathfrak{p}_5), \quad (c_{10}, \mathfrak{p}_{10}), \quad (c_{11}, \mathfrak{p}_{11}).
\]
}
\end{Example}

\subsection{Deformations of Hom-Lie superalgebras} The deformation theory of Hom-Lie superalgebras in characteristic 2 will be discussed here. As a result, we also cover the Lie case, namely $\alpha=\mathrm{Id}_\fg$. Over a field of characteristic zero, the study has been carried out in \cite{AMS, TR}. 

Let $(\fg,[\cdot, \cdot],s,\alpha)$ be a Hom-Lie superalgebra over a field $\mathbb K$ of characteristic 2. A deformation of $\fg$ is a family of Hom-Lie superalgebras $\fg_t$ specializing in $\fg$ when the parameter $t$ equals $0$ and  where the Hom-Lie superalgebra structure is defined on the tensor product $\fg \otimes {\mathbb K}[[t]]$ when $\fg$ is finite dimensional. 
The bracket in the deformed Hom-Lie superalgebra $\fg_t$ is a $\K[[t]]$-bilinear map of the form (for all $x,y\in \fg$):
 \[
 \begin{array}{lcl}
  \displaystyle  [x,y]_t & = & \displaystyle [x,y]+ \sum_{i\geq 1} c_i(x,y)t^i,
 \end{array}
 \]
 while the squaring $s_t$, with respect to $\K[[t]]$, on the Hom-Lie superalgebra $\fg_t$ is   given by (for all $x\in \fg_\od$):
 \[
 \begin{array}{lcl}
\displaystyle    s_t(x)  & =  & \displaystyle s(x)+   \sum_{i\geq 1} {\mathfrak p}_i(x)t^i,
 \end{array}
 \]
 where $(c_i, \mathfrak{p}_i) \in XC^2_{\alpha,\ev}(\fg;\fg)$ for all $i\geq 1$. We will assume that $c_0(\cdot,\cdot)=[\cdot, \cdot]$ and ${\mathfrak p}_0(\cdot)=s(\cdot)$.
 
 According to deformation theory, we call a deformation {\it  infinitesimal} if the bracket $[\cdot, \cdot]_t$ and the squaring $s_t(\cdot)$ define a Hom-Lie superalgebra structure$\mod (t^2)$ (degree~1), that is $[\cdot,\cdot]_t  = [\cdot,\cdot]+  c_1(\cdot,\cdot)t$ and $s_t(\cdot)   = s(\cdot)+    {\mathfrak p}_1(\cdot)t$.  A deformation is said to be of    {\it  order} $n$ if the bracket $[\cdot, \cdot]_t$ and the squaring $s_t(\cdot)$ define a Hom-Lie superalgebra structure$\mod (t^{n+1})$, that is 
 \[
  [\cdot,\cdot]_t  =  [\cdot,\cdot]+ \sum_{1\leq i\leq n}  c_i(\cdot,\cdot)t^i
\text{ and } s_t(\cdot)   =   s(\cdot)+   \sum_{1\leq i\leq n} {\mathfrak p}_i(\cdot)t^i.
 \]
 Afterwards, one tries to extend a deformation of order $n$ to a  deformation of order $n+1$. All obstructions are cohomological, as we will see.

 \begin{Theorem}\label{thmdefor}
 Let  $(\fg, [\cdot, \cdot], s, \alpha)$ be  a Hom-Lie superalgebra in characteristic 2 and $(\fg_t, [\cdot, \cdot]_t, s_t, \alpha)$ be a deformation. Assume that $(c_1, \mathfrak{p}_1)\neq (0,0)$.  Then 
 
 \textup{(}i\textup{) }$(c_1, \mathfrak{p}_1)$ is a 2-cocycle, i.e. $(c_1, \mathfrak{p}_1)  \in Z_\alpha^2(\fg; \fg)$.

\textup{(}ii\textup{) } For $n>1$, consider the following maps:
\[
\begin{array}{lcl}
 C_n(x,y,z) & := & \displaystyle \sum_{\substack{i+j=n\\i,j<n}} c_i(c_j(x,y),\alpha(z))+\circlearrowleft (x,y,z), \quad \text{for all $x,y, z\in \fg$,}
\\[5mm]
Q_n(x,y) & := &\displaystyle \sum_{\substack{i+j=n\\i,j<n}} c_i(\mathfrak{p}_j(x),\alpha(y)) +
\displaystyle \sum_{\substack{i+j=n\\i,j<n}} c_i(c_j(x,y),\alpha(x)) , \; \; \text{for all $x\in \fg_\od$ and $y\in \fg$.}
\end{array}
\]
A  deformation of order $n-1$ can be extended to a deformation of order $n$ if and only there exists $(c_n, \mathfrak{p}_n)$
\[
(C_n, Q_n)= \mathfrak{d}_\alpha^2(c_n, \mathfrak{p}_n).
\]
 \end{Theorem}
 
 \begin{proof}
 (i) Checking that $c_1$ satisfies the condition \eqref{2-cocCond1alpha} is a routine, see \cite{AMS}. Let us deal with the squaring  $s_t$.
 We have
 \begin{equation}\label{sJIt1}
 \begin{array}{l}
  [s_t(x), \alpha(y)]_t  =  \displaystyle [  s(x)+   \sum_{i\geq 1} {\mathfrak p}_i(x)t^i, \alpha(y)]_t \\[1mm]
 =  \displaystyle [s(x),\alpha(y)]+ \sum_{i\geq 1} c_i(s(x), \alpha(y))t^i + \sum_{i\geq 1}[{\mathfrak p}_i(x), \alpha(y)]t^i + \sum_{i,j\geq 1}c_j({\mathfrak p}_i(x), \alpha(y))t^{i+j}.
 \end{array}
 \end{equation}
 On the other hand,
 \begin{equation}\label{sJIt2}
  \begin{array}{lcl}
 [\alpha(x), [x,y]_t]_t & = &   \displaystyle [ \alpha(x), [x,y]+ \sum_{i\geq 1} c_i(x,y)t^i ]_t\\[2mm]
  & = & \displaystyle [ \alpha(x), [x,y]] + \sum_{i\geq 1} c_i(\alpha(x), [x,y])t^i + \sum_{i\geq 1} [\alpha(x),c_i(x,y)]t^i +  \displaystyle\\[2mm]
  &&+ \sum_{i,j\geq 1} c_i(\alpha(x),c_j(x,y))t^{i+j}.
  \end{array}
 \end{equation}
 Collecting the coefficient of $t$ in the condition
 $
[s_t(x), \alpha(y)]_t =[\alpha(x), [x,y]_t]_t,$ we get 
 \[
 c_1(s(x), \alpha(y))+ [\alpha(y), {\mathfrak p}_i(x)]+c_1(\alpha(x), [x,y])+ [\alpha(x),c_1(x,y)]=0,
 \]
 which corresponds to Condition \eqref{2-cocCond2alpha}.
 Therefore, $(c_1, {\mathfrak p_1})$ is a 2-cocycle on $\fg$ with values in the adjoint representation.

(ii) 
Let us first show that the pair $(C_n,Q_n)$ is a cochain in $XC^3(\fg,\fg)$. Indeed,
\begin{eqnarray*}
   && Q_n(x_1+x_2,y)=\sum_{i,j}c_i(\mathfrak{p}_j(x_1+x_2),\alpha(y))+c_i(c_j(x_1+x_2,y),\alpha(x_1+x_2))\\
    &&= \sum_{i,j}c_i(\mathfrak{p}_j(x_1)+\mathfrak{p}_j(x_2) +c_j(x_1,x_2),\alpha(y))+c_i(c_j(x_1,y)+c_j(x_2,y),\alpha(x_1)+\alpha(x_2))\\
    &&=Q_n(x_1,y)+Q_n(x_2,y)+\sum_{i,j}(c_i(c_j(x_1, x_2),\alpha(y))+c_i(c_j(x_1, y),\alpha(x_2))+c_i(c_j(x_2, y),\alpha(x_1)) )\\
    &&=Q_n(x_1,y)+Q_n(x_2,y)+C_n(x_1, x_2, y).
\end{eqnarray*}


Collecting coefficients of $t^n$ in \eqref{2-cocCond1alpha} leads to  $C_n=d^2_\alpha c_n$, see \cite{AMS}. Let us deal with the squaring. Consider  the coefficient of $t^n$ in the condition $[s_t(x), \alpha(y)]_t =[\alpha(x), [x,y]_t]_t,$ and using  Eqns. (\ref{sJIt1}, \ref{sJIt2}), we get (for all $x\in \fg_\od$ and $y\in \fg$):
\[\begin{array}{lcl}
c_n(s(x), \alpha(y))+[{\mathfrak p}_n(x), \alpha(y)]+\displaystyle \sum_{\substack{i+j=n \\i,j<n}} c_j( {\mathfrak p}_i(x), \alpha(y))&=& c_n(\alpha(x),[x,y])+[\alpha(x), c_n(x,y)]\\[2mm]
&&\displaystyle +\sum_{\substack{i+j=n\\i,j<n}} c_i(\alpha(x), c_j(x,y)).
\end{array}
\]
Let us rewrite this expression. We get
\[
d^2_\alpha  {\mathfrak p}_n(x,y)= \sum_{\substack{i+j=n\\i,j<n}} c_i(\alpha(x), c_j(x,y))+\displaystyle \sum_{\substack{i+j=n \\i,j<n}} c_j( {\mathfrak p}_i(x), \alpha(y)).
\]
Therefore, $(C_n, Q_n)=(d^2_\alpha c_n, d^2_\alpha \mathfrak{p}_n)= {\mathfrak d}^2_\alpha (c_n,\mathfrak{p}_n).$
\end{proof}
Now, we discuss equivalent deformations. 
\begin{Definition}\rm{
Let  $(\fg, [\cdot, \cdot], s, \alpha)$ be  a Hom-Lie superalgebra in characteristic 2. Let  $(\fg_t, [\cdot, \cdot]_t, s_t, \alpha)$ and  $(\tilde{\fg}_t, \widetilde{[\cdot, \cdot]}_t, \tilde{s}_t, \alpha)$  be two deformations of $\fg$, such that $\tilde{[\cdot, \cdot]}_0=[\cdot, \cdot]_0=[\cdot, \cdot]$ and  $\tilde{s}_0=s_0=s$.   We say that the two deformations $\fg_t$ and $\tilde \fg_t$  are equivalent if there exists a $\K[[t]]$-linear map $\tau:\fg_t\rightarrow \tilde \fg_t$ of the form $\tau (a) = \mathrm{id}_\fg(a)+
\sum_{i\geq 1}
\tau_i(a)t^i$ for all  $a\in \fg$, that is an isomorphism of Hom-Lie superalgebras.}
\end{Definition}
\begin{Theorem}\label{thmdefor2}
 Two 1-parameter formal deformations $\fg_t$ and $\tilde \fg_t$ of $\fg$ given by the collections $(c,{\mathfrak p})$ and $(\tilde c, \tilde{\mathfrak p})$ are equivalent through  an isomorphism of the form $\tau  = \mathrm{id}_\fg+
\sum_{i\geq 1}
t^i \tau_i$  if and only if $\tau$ links $(c,{\mathfrak p})$ and $(\tilde c, \tilde{\mathfrak{p}})$ by the following formulae \textup{(}for all $n > 0$\textup{)}:
\begin{equation}\label{equivC}
\sum_{
i+j=n}
\tau_i(\tilde c_j(x, y)) =
\sum_{
i+j+k=n}
c_i(\tau_j(x), \tau_k(y)), \quad \text{for all $x,y\in \fg$}, 
\end{equation}
and \textup{(}for all $x\in \fg_\od$\textup{)}:
\begin{equation}
    \label{equivS}
\displaystyle \sum_{i+j= n} \tau_i(\tilde{\mathfrak p}_{j}(x))=\displaystyle \sum_{i+j=n }c_i(x,\tau_{j}(x))+\sum_{2i+j=n}{\mathfrak p}_j(\tau_i(x))+
\displaystyle
\sum_{\substack{u+v+j=n\\ 1\leq u<v\\1 \leq j}}c_j([\tau_u(x), \tau_v(x)]).
\end{equation}
In particular, if $n=1$ we get 
\[
\begin{array}{lcl}
{\tilde c}_1(x,y) & = & c_1(x,y)+\tau_1([x,y])+[x,\tau_1(y)]+[y,\tau_1(x)], \quad \text{for all $x,y\in \fg$},\\[2mm]
\tilde{\mathfrak p}_1(x) & = & \mathfrak p_1(x)+\tau_1(s(x))+ [x,\tau_1(x)], \quad \text{for all $x\in \fg_\od$}.
\end{array}
\]
Hence, $({\tilde c}_1,\tilde{\mathfrak p}_1)$ and $( c_1,\mathfrak p_1)$ are in the same cohomology class.

\end{Theorem}
 \begin{Corollary} Infinitesimal deformations are classified by the cohomology group $\mathrm{H}^2_\alpha(\fg; \fg)$. 
 \end{Corollary}
\begin{proof}[Proof of Theorem \ref{thmdefor2}]
Checking Eq. \eqref{equivC} is a routine, see \cite{AMS}. Let us check Eq. \eqref{equivS}. 
 
 We have
 \begin{align*}
 &\tau(\tilde{s}_t(x))=\tau \left (s(x)+\sum_{i\geq 1} \tilde{\mathfrak p}_i(x)t^i \right ) =s(x)+\sum_{i\geq 1}\tilde{\mathfrak p}_i(x)t^i+\sum_{i,j\geq 1}\tau_j(\tilde{ \mathfrak p}_i(x))t^{i+j}+\sum_{j\geq 1} \tau_j(s(x))t^j\\
 &=\sum_{i,j\geq 0}\tau_j(\tilde{ \mathfrak p}_i(x))t^{i+j}
 \end{align*}
 On the other hand, we get 
 \[
 \begin{array}{lcl}
 s_t(\tau(x)) & = & \displaystyle s\left (x+\sum_{i\geq 1} \tau_i(x)t^i \right )+\sum_{i\geq 1} {\mathfrak p}_i \left (x+\sum_{i\geq 1}\tau_i(x)t^i \right )t^j\\[3mm]
 &= & \displaystyle s(x)+\sum_{i\geq 1} s(\tau_i(x))t^{2i}+\sum_{i\geq 1}[x,\tau_i(x)]t^i+\sum_{i\geq 1} {\mathfrak p}_j \left (x+\sum_{i\geq 1}\tau_i(x)t^i \right )t^j\\[2mm]
 &= & \displaystyle s(x)+\sum_{i\geq 1} s(\tau_i(x))t^{2i}+\sum_{i\geq 1}[x,\tau_i(x)]t^i+\sum_{i\geq 1} {\mathfrak p}_j (x)t^j+\sum_{i\geq 1}{\mathfrak p}_j(\tau_i(x)) t^{2i+j}\\[2mm]
 & & \displaystyle +\sum_{i,j\geq 1} c_j(x, \tau_i(x)) t^{i+j} + \displaystyle \sum_{1 \leq j} \sum_{\substack{1 \leq u<v}}c_j(\tau_u(x), \tau_v(x))t^{j+u+v}\\[3mm]
 & =& \displaystyle \sum_{i,j\geq 0}{\mathfrak p}_j(\tau_i(x)) t^{2i+j}+ \displaystyle +\sum_{i,j\geq 0} c_j(x, \tau_i(x))t^{i+j}+ \displaystyle \sum_{1 \leq j} \sum_{\substack{1 \leq u<v}}c_j(\tau_u(x), \tau_v(x))t^{j+u+v}
 \end{array}
 \]
The result follows by identification. 
\end{proof}

\end{document}